\documentclass[11pt]{article}

\usepackage[margin=1in]{geometry}

\usepackage{algorithm,algorithmic}
\usepackage{amsmath}
\usepackage{amssymb}
\usepackage{amsthm}
\usepackage{bm}
\usepackage{color}
\usepackage{graphicx}
\usepackage{mathtools}
\usepackage{hyperref}
\usepackage{enumerate}
\usepackage[shortlabels]{enumitem}
\usepackage[colorinlistoftodos,prependcaption,textsize=footnotesize]{todonotes}

\usepackage{hyperref}
\hypersetup{
    colorlinks=true,     
    linkcolor=blue,      
    citecolor=blue,      
    filecolor=blue,      
    urlcolor=blue        
}

\newtheorem{lemma}{Lemma}[section]

\newtheorem{theorem}[lemma]{Theorem}
\newtheorem{corollary}[lemma]{Corollary}
\newtheorem{proposition}[lemma]{Proposition}
\newtheorem{remark}[lemma]{Remark}
\newtheorem{example}[lemma]{Example}

\newcommand{\Aut}{\operatorname{Aut}}
\newcommand{\face}{\mathrel{\unlhd}}
\newcommand{\minFace}{\mathrm{Face}}
\newcommand{\spanVec}{\ensuremath{\mathrm{span}\,}}

\newcommand{\reInt}{\mathrm{ri}\,}

\newcommand{\norm}[1]{\lVert{#1}\rVert}
\newcommand{\T}{*} 
\newcommand{\inProd}[2]{\langle #1 , #2 \rangle }

\newcommand{\cS}{\mathcal{S}}

\newcommand{\cH}{\mathcal{H}}

\newcommand{\psdcone}[1]{\cS_+^{#1}}
\newcommand{\psdG}[1]{\cS_+({#1})}

\renewcommand{\H}{\cH}

\newcommand{\RR}{\mathbb{R}}

\newcommand{\SC}{\operatorname{SC}}

\DeclareMathOperator{\rank}{rank}
\newcommand{\tr}{\textup{tr}}

\newcommand{\stdCone}{ {\mathcal{K}}}
\newcommand{\stdFace}{{\mathcal{F}}}
\newcommand{\ambSpace}{ \mathcal{E} }

\renewcommand{\Re}{\mathbb{R}}   
\newcommand{\alg}{\mathcal{A}}
\newcommand{\algD}{\mathcal{B}}
\newcommand{\UT}{\mathcal{T}}
\newcommand{\LT}{\mathcal{L}}
\newcommand{\HE}{\mathcal{H}}
\newcommand{\stdHC}{ {\mathcal{K}(\alg)}}

\newcommand{\id}{{e}}
\newcommand{\closure}{\mathrm{cl}\,}

\title{Faces of homogeneous cones and applications to homogeneous chordality}
\numberwithin{equation}{section}
\author{
    Jo\~ao Gouveia\thanks{CMUC, Department of Mathematics, University of Coimbra, 3001-454 Coimbra, Portugal.
    This author was supported by the Centre for Mathematics of the University of Coimbra (UIDB/00324/2020,
    funded by the Portuguese Government through FCT/MCTES).
    During part of this work, the author also held a visiting professorship at the Institute of Statistical Mathematics, Japan.
    Email: \href{jgouveia@mat.uc.pt }{jgouveia@mat.uc.pt}.}
	\and 
	Masaru Ito%
	\thanks{Department of Mathematics, College of Science and Technology, Nihon University,
		1-8-14 Kanda-Surugadai, Chiyoda-Ku, Tokyo 101-8308, Japan.
		This author
		was supported partly by the JSPS Grant-in-Aid for Early-Career Scientists 	21K17711.
		Email: \href{ito.masaru@nihon-u.ac.jp}{ito.masaru@nihon-u.ac.jp}.}
 \and
	Bruno F. Louren\c{c}o\thanks{Department of Fundamental Statistical Mathematics, Institute of Statistical Mathematics, Japan.
	This author was supported partly by the JSPS Grant-in-Aid for Early-Career Scientists  23K16844.
	Email: \href{bruno@ism.ac.jp}{bruno@ism.ac.jp}.}
}

\begin{document}	
\maketitle
\begin{abstract}
A convex cone $\stdCone$ is said to be homogeneous if its group of automorphisms acts transitively on its relative interior. 
Important examples of homogeneous cones include symmetric cones and cones of positive semidefinite (PSD) matrices that follow a sparsity pattern given by a homogeneous chordal graph.
Our goal in this paper is to elucidate  the facial structure of homogeneous cones and make it as transparent as the faces of the PSD matrices.
We prove  that each face of a homogeneous cone $\stdCone$ is mapped by an automorphism of $\stdCone$ to one of its finitely many so-called \emph{principal faces}.
Furthermore, constructing such an automorphism can be done algorithmically by making use of a generalized Cholesky decomposition.
Among other consequences, we give a proof that homogeneous cones are projectionally exposed, which strengthens the previous best result that they are amenable.
Using our results, we will carefully analyze the facial structure of  cones of PSD matrices satisfying homogeneous chordality and discuss consequences for the corresponding family of PSD completion problems. 
 
\end{abstract}

%
%
%
%
%
%
%
%

\section{Introduction}
A convex cone is said to be \emph{homogeneous} if its group of automorphisms acts transitively on its relative interior. 
Conic programming over homogeneous cones is a cornerstone of the modern optimization landscape as it contains as special cases semidefinite programming (i.e., optimization over positive semidefinite matrices) and linear programming (i.e., optimization over the nonnegative orthant).
However, the positive semidefinite matrices and nonnegative orthants are only merely two very special examples of homogeneous cones, as they are also self-dual under an appropriate choice of inner product and, therefore, are \emph{symmetric cones}.

An important class of homogeneous but not necessarily symmetric cones include the positive semidefinite matrices with certain sparsity patterns.
Let $G$ be a graph on $n$ vertices and consider the cone
$\psdG{G}$ of $n\times n$ real symmetric positive semidefinite (PSD) matrices whose sparsity pattern follows $G$, i.e., 
$ x \in \psdG{G}$ if and only if
$x$ is an $n\times n$ PSD matrix with
$x_{ij} = 0$ for all
$i \neq j$ such that $(i,j)$ is not an edge of $G$.
The case where $G$ is a chordal graph has been extensively studied in the literature as it is typically more efficient to solve problems exploiting  chordal sparsity instead of seeing it simply as a SDP instance, e.g., \cite{FKMN01,NFFKM03}. 
The dual of $\psdG{G}$ is also important, since it corresponds to the PSD completable matrices with pattern given by $G$. 

If $G$ is not only chordal but has no induced subgraph that is a path on four vertices, then $\psdG{G}$ becomes a homogeneous cone \cite{Ishi13} and such graphs were called \emph{homogeneous chordal} in \cite{TV23}. PSD matrices that follow a homogeneous chordal 
sparsity pattern have remarkable properties that may fail  in general for chordal graphs.
A notable example is that, under an appropriate ordering of the vertices, inverses of  Cholesky factors still respect the sparsity pattern defined by $G$, see \cite[Theorem~3.1]{TV23} for more details.
Therefore, for homogeneous chordal graphs, $\psdG{G}$ and its dual form important classes of homogeneous cones that have practical relevance and potential for interesting applications. 

Motivated by the particular case of homogeneous chordality, our goal in this paper is to elucidate the facial structure of general homogeneous cones and make it as clear as the facial structure of  PSD matrices, which we will now briefly review.
Let $\psdcone{n}$ denote the set of real $n\times n$ symmetric positive semidefinite matrices.
If $\stdFace$ is a face of $\psdcone{n}$, then there exists exists a $n\times n$ orthogonal matrix $q$ and $r \leq n$ such that
\begin{equation}\label{eq:psd_faces}
q\stdFace q^\T = \left\{ \begin{pmatrix} a& 0\\ 0 & 0\end{pmatrix} \mid a \in \psdcone{r} \right\},
\end{equation}
where $q^\T$ denotes the transpose of $q$,
see \cite[Example~3.2.2]{pataki_handbook}, \cite[Section~6]{BC75}. That is, each face of $\psdcone{n}$ is linearly isomorphic to a smaller PSD cone and this isomorphism  can be realized as an automorphism of $\psdcone{n}$ since the map $Q$ that takes $x$ to $qxq^\T$ satisfies $Q(\psdcone{n}) = \psdcone{n}$.
The matrix $q$ does not come from thin air: 
the kernel of matrices in the relative interior of $\stdFace$ is unique and correspond to some fixed subspace $V \subseteq \Re^n$. 
With that, $\stdFace = \{x \in \psdcone{n} \mid \ker x \supseteq V\}$ holds.
Therefore, computing $q$ is an entirely constructive endeavour, as it is enough to pick any $x \in \reInt \stdFace$ and let $q$ be such that its rows are orthogonal and the last $n-r$ rows form an orthonormal basis for $\ker x$.

Another interesting property is that 
if we let $v \coloneqq  q^\T\begin{psmallmatrix} e_r & 0\\ 0 & 0\end{psmallmatrix} q $ , where $e_r$ is the $r\times r$ identity matrix and define the map $\mathcal{P}(x) \coloneqq vxv$, we have that 
$\mathcal{P}$ is a projection (i.e., $\mathcal{P}^2 = \mathcal{P}$) satisfying 
$\mathcal{P}(\psdcone{n}) = \stdFace$.
In this way, not only the faces of $\psdcone{n}$ are isomorphic to smaller PSD cones, they are also projected versions of it.
Cones for which each face arises as a projection of the original cone are known as \emph{projectionally exposed} and were first described by Borwein and Wolkowicz in the context of their facial reduction approach in \cite{BW81}.
One of the main goals of this paper is to prove analogous results to
 \eqref{eq:psd_faces} for homogeneous cones and discuss their consequences.

Another motivation for this work comes from the study of facial exposedness in general cones. 
We recall that a closed convex cone $\stdCone$ is said to be \emph{facially exposed} if every face arises as the intersection of $\stdCone$ with one of its supporting hyperplanes. 
Although facial exposedness is useful, certain applications require stronger facial exposedness properties such as \emph{niceness} (also known as facial dual completeness) \cite{Pa07,Pa13,Pa13_2}, \emph{amenability} \cite{L21,LRS20} or the aforementioned
{projectional exposedness} \cite{BW81,ST90}.

Every projectionally exposed (p-exposed) cone is amenable, every amenable cone is nice and every nice cone is facially exposed. 
In dimension at most three, facial exposedness implies p-exposedness, see \cite[Theorem~3.2]{PL88} or \cite[Theorem~4.6]{ST90}.
In dimension four, there exists a facially exposed cone that is not amenable \cite{Ve14} and a nice cone that is not amenable \cite{LRS20}. 

Regarding homogeneous cones, Truong and Tun\c{c}el showed that they are facially exposed \cite[Theorem~6]{TT03}. Later, 
Chua and Tun\c{c}el showed that homogeneous cones are nice 
\cite[Proposition~4 and Section~4.2]{CT06}.
Then, in \cite{LRS20} it was shown that homogeneous cones are amenable.
In this work, we take another step and show that homogeneous cones are, in fact, p-exposed. 
In this way, homogeneous cones form a large family of p-exposed cones and it is not currently known if there are any other interesting classes of p-exposed cones strictly containing
homogeneous cones.
Notably, homogeneous cones are also hyperbolicity cones \cite[Section~8]{Gu97}, which is another important class of 
cones \cite{Ga59,Re06}. 
However, the strongest result so far is that hyperbolicity cones are amenable \cite{LRS23} and it is not known whether they are p-exposed in general.

Our main results are as follows.
\begin{enumerate}[$(i)$]
	\item We show that results analogous to \eqref{eq:psd_faces} hold for general homogeneous cones, see Theorem~\ref{th:hface}. 
	In particular, given a homogeneous cone $\stdCone$ of rank $r$ (see Section~\ref{sec:hom}), there are $2^r$ distinguished faces called \emph{principal faces} that play a role similar to the right-hand-side of \eqref{eq:psd_faces} in the sense that
	every face $\stdFace$ of $\stdCone$ can be mapped to a principal face via an automorphism of $\stdCone$. 
	Similarly to the case of PSD matrices, the automorphism can be explicitly constructed given any $x \in \reInt \stdFace$ by making use of a generalized Cholesky factorization (Algorithm~\ref{alg:cholesky}).
	
	This will allow us to easily obtain certain results that are intuitive but are not trivial to extract from the existing literature. 
	For example, Theorem~\ref{th:hface} contains as a particular consequence the fact that proper faces of homogeneous cones are also homogeneous cone themselves and must have strictly smaller rank, which, as far as we know, has never been formally stated or verified previously.
	One application of this is the computation of the length of a longest chain of faces of homogeneous cones, see Corollary~\ref{col:chain}.
	We will also show that conjugate faces have complementary ranks, see Proposition~\ref{prop:conj_rank}.
	
	\item Among the consequences of our discussion, of particular note is the fact that homogeneous cones are projectionally exposed, which strengthens previous results in \cite{TT03,CT06,LRS20}. 
	
	\item We discuss  in detail the facial structure of PSD matrices with sparsity pattern given by a homogeneous chordal graph $G$, see Section~\ref{sec:chord}. 	
	In particular we show that every face of $\psdG{G}$ is isomorphic to some face that arises by considering induced subgraphs of $G$ and this isomorphism can be realized by an automorphism of $\psdG{G}$, see Theorem~\ref{theo:h_chord_faces}.
	We will also prove analogous results to the dual cone 
	of $\psdG{G}$ which corresponds to the PSD completable matrices determined by $G$.
	Then, we will check that when  Algorithm~\ref{alg:cholesky} is specialized to 
	the dual of $\psdG{G}$, it will allow us to compute PSD completions having certain desirable properties, see 
	Theorem~\ref{theo:completion}.
	In Remark~\ref{rem:failure} we will see that some of our results may fail if a graph is chordal but not homogeneous chordal.
	
\end{enumerate}

This work is divided as follows. In Section~\ref{sec:prel}, we recall some properties of convex sets and its faces and then we discuss the T-algebraic framework of Vinberg for handling homogeneous cones. Section~\ref{sec:faces} contains the main results of this paper regarding the facial structure of homogeneous cones. 
In Section~\ref{sec:chord}, we discuss several applications of the results in Section~\ref{sec:faces} to the facial structure of $\psdG{G}$ and its dual.
We conclude this work in Section~\ref{sec:conc} with a list of open questions and some remarks on how our discussion can be useful in facial reduction approaches.
 



\section{Preliminaries}\label{sec:prel}

Let $\ambSpace$ be a finite dimensional Euclidean space equipped with an inner product $\inProd{\cdot}{\cdot}$ and corresponding induced norm $\norm{\cdot}$. Let $C \subseteq \ambSpace$ be a convex set. 
We denote its relative interior and closure by $\reInt C$ and $\closure C$, respectively. The smallest subspace of $\ambSpace$ containing $C$ is denoted by $\spanVec C$.

Let $\stdCone \subseteq \ambSpace$ be a convex cone.
We say that $\stdCone$ is \emph{pointed} if 
$\closure \stdCone \cap - \closure \stdCone = \{0\}$ holds.
$\stdCone$ is said to be \emph{full-dimensional} if $\dim \stdCone = \dim \ambSpace$ or, equivalently, if $\stdCone$ has non-empty interior. A pointed full-dimensional cone is said to be \emph{regular}.
Two convex cones $\stdCone_1, \stdCone_2$ are said to be \emph{linearly isomorphic} if there exists a linear bijection $A$ such that $A(\stdCone_1) = \stdCone_2$ holds.

A convex cone $\stdFace \subseteq \stdCone$ is said to be a \emph{face} of $\stdCone$ if $x,y \in \stdCone$ and $x + y \in \stdFace$ implies that $x,y \in \stdFace$. In this case, we write $\stdFace \face \stdCone$ and, by convention, we will only consider non-empty faces. 
A face $\stdFace$ is said to be \emph{proper} if $\stdFace \neq \stdCone$ and it is \emph{maximally proper} if there is no 
proper face $\hat \stdFace \face \stdCone$ such that 
$\stdFace \subseteq \hat \stdFace $ and $\stdFace \neq \hat \stdFace$. An \emph{extreme ray} of $\stdCone$ is a face of the form $\stdFace = \{\alpha x \mid \alpha \geq 0\}$ for some nonzero $x \in \stdCone$. 
In this case, we say that \emph{$x$ generates an extreme ray of $\stdCone$}.


Next, let $x \in \stdCone$. We denote 
by $\minFace(x,\stdCone)$ the minimal face of $\stdCone$ containing $x$, i.e., $\minFace(x,\stdCone) = \bigcap _{\stdFace \face \stdCone, x \in \stdFace} \stdFace$. 
Let $\stdFace \face \stdCone$ be a face and $x \in \stdCone$, then 
\begin{equation}\label{eq:minF}
\stdFace = \minFace(x,\stdCone) \Leftrightarrow x \in \reInt \stdFace,
\end{equation}
e.g., see \cite[Proposition~3.2.2]{pataki_handbook} or \eqref{eq:minF} can also be inferred by the discussion in \cite{Ba73}.


The dual cone of $\stdCone$ is denoted by $\stdCone^* \coloneqq \{y \in \ambSpace \mid \inProd{x}{y} \geq 0, \forall x \in \stdCone\}$. 
A face $\stdFace \face \stdCone$ is said to be \emph{exposed} if there exists $y \in \stdCone^*$  such that $\stdFace = \stdCone \cap \{y\}^\perp$ holds.
In this case, we say that \emph{$y$ exposes the face $\stdFace$}.
The conjugate face of $\stdFace$ is defined as $\stdFace^\Delta \coloneqq \stdCone^* \cap \stdFace^\perp$ and satisfies
\begin{equation}\label{eq:face_conj}
\stdFace^\Delta = \stdCone^* \cap \{x\}^\perp, \qquad \forall x \in \reInt \stdFace,
\end{equation}
which follows from \cite[Theorem~6.4]{rockafellar}.
In view of \eqref{eq:minF} and \eqref{eq:face_conj}, if $\stdFace$ is exposed by $y$, then $\stdFace$ is precisely the face that is conjugated to $\minFace(y, \stdCone^*)$. 
We will also need a well-known lemma that connects extreme rays and maximal proper faces via conjugacy.
\begin{lemma}\label{lem:conj}
If $\stdFace \face \stdCone$ is an exposed extreme ray, then $\stdFace^\Delta$ is a maximal proper face of $\stdCone^*$.
\end{lemma}
\begin{proof}
In this proof we need the well-known fact that a face is exposed if and only if $\stdFace^{\Delta \Delta} = \stdFace$ holds.
In addition, all maximal proper faces are exposed, e.g., see \cite[Corollary~2.2]{Tam85}.
With these facts in mind, suppose that $\hat \stdFace \face \stdCone^*$ is a maximally proper face containing
$\stdFace^\Delta$.
As conjugacy inverts inclusion, we 
have $\hat \stdFace^{\Delta} \face \stdFace^{\Delta \Delta} = \stdFace $, so $\hat \stdFace ^\Delta = \{0\}$ or $\hat \stdFace ^\Delta = \stdFace$.
In the former case we have $\hat \stdFace = \hat \stdFace^{\Delta \Delta} = \stdCone^*$, which contradicts the properness of $\hat \stdFace$. So, we must be in the latter case and $\hat \stdFace = \hat \stdFace^{\Delta \Delta} = \stdFace^{\Delta}$ holds.
\end{proof}

A face $\stdFace \face \stdCone$ is said to be \emph{projectionally exposed} (or p-exposed) if there exists a linear map $\mathcal{P}: \ambSpace \to \ambSpace$  such that 
$\mathcal{P}(\stdCone) = \stdFace$ and $\mathcal{P}^2 = \mathcal{P}$.
If all faces of $\stdCone$ are p-exposed, then $\stdCone$ is said to be p-exposed. 
P-exposedness was proposed in \cite{BW81} and 
a comprehensive discussion on p-exposedness is given in \cite{ST90}.

If the projection $\mathcal{P}$ can be taken to be self-adjoint with respect to $\inProd{\cdot}{\cdot}$, then 
$\stdFace$ is said to be \emph{orthogonally projectionally exposed} (or o.p.-exposed). Clearly, this is a notion that depends on the choice of inner product, so to emphasize this choice, we will sometimes say that \emph{$\stdFace$ is o.p.-exposed under $\inProd{\cdot}{\cdot}$}.
If all faces of $\stdCone$ are o.p.-exposed with respect the same inner product, then $\stdCone$ is said to be o.p.-exposed.
Being o.p.-exposed is a more restrictive property and, for example, an o.p.-exposed polyhedral cone must 
be linearly isomorphic to a nonnegative orthant, see
\cite[Theorem~3.7 and pg.~233]{ST90} and \cite{BLP87}.
That said, all symmetric cones are o.p.-exposed under an appropriate inner product \cite[Proposition~33]{L21}.



\subsection{Homogeneous cones and T-algebras}\label{sec:hom}
Let $\stdCone \subseteq \ambSpace$ be a convex cone. The automorphisms of $\stdCone$, denoted by $\Aut(\stdCone)$, are the group of linear bijections $Q: \ambSpace \to \ambSpace$ satisfying $Q(\stdCone) = \stdCone$, where the group operation is the function composition.
We say that $\stdCone$ is \emph{homogeneous} if $\Aut(\stdCone)$ acts transitively on the relative interior of $\stdCone$, i.e., 
for every $x,y \in \reInt \stdCone$, there must exist $Q \in \Aut(\stdCone)$ satisfying $Q(x) = y$. 

\emph{Symmetric cones} are homogeneous cones that are self-dual under some inner product and the theory of Euclidean Jordan algebras \cite{K99,FK94} seems to be the standard algebraic framework to deal with them. 
For homogeneous cones,  there are different, albeit closely related, algebraic frameworks described in many foundational works \cite{V63,Ro66,Gin92}. 
Brief overviews of this state of affairs can be seen in \cite[Section~1]{YN15} and \cite[Section~6]{TV23}.

Our goal in this paper is to describe the facial structure of homogeneous cones, which as far as we know, has not been described explicitly before. 
Nevertheless, there are previous works describing polynomials that are closely related to the boundary structure of homogeneous cones \cite{Ishi01,Na14,Na18,GIL24,Na24}. 
A previous work by Ishi also contains important information on the action of the automorphism group on the boundary of the cone, see \cite{Ishi00}.
We will revisit this point later in Section~\ref{sec:faces}.
In optimization, the study of homogeneous cones seems to have been initially motivated by self-concordant barriers and interior-point methods \cite{Gu96,GT98,CH09}, with other works focusing on geometric and representational aspects \cite{TX01,CH03,TT03,CT06}.

%
In this work, we will use the theory of T-algebras of Vinberg \cite{V63} as it seems to be the most natural for our purpose. We will later justify this choice and discuss related approaches  in Section~\ref{sec:other}.
Unfortunately, as of this writing, although the original russian text of \cite{V63} is freely available on the internet, it is not completely trivial to get a copy of the English translation. 
Because of that, as we review T-algebras, we refer extensively to more accessible references in English such as \cite{CH03,CT06,CH09,KTX12}. 

To start, we say that a \emph{matrix algebra of rank $r$} is an algebra $\alg$ over $\Re$ equipped with a bigradation, i.e., a decomposition as a direct sum $\alg = \bigoplus _{i,j=1}^r \alg _{ij}$ where the $\alg _{ij}$ are subspaces  satisfying the following properties:
\begin{equation}\label{eq:m_rules}
\begin{aligned}
\alg _{ij} \alg _{jk} & \subseteq \alg _{ik} \\
\alg _{ij} \alg _{kl} & = \{0\} \qquad \text{if } j \neq k.  
\end{aligned} 
\end{equation}
Given $a \in \alg$, we can write $a = \sum _{i,j=1}^r a_{ij}$ in a unique way with $a_{ij} \in \alg_{ij}$.  
We will refer to $a_{ij}$ as the $(i,j)$ component of $a$.
Multiplication in $\alg$ is analogous to the usual matrix multiplication since \eqref{eq:m_rules} implies that
\begin{equation}\label{eq:m_mult}
(ab)_{ij} =  \sum _{k=1}^r a_{ik} b_{kj}
\end{equation}
holds for $a,b \in \alg$.
A \emph{matrix algebra with involution} is a matrix algebra equipped with a linear bijection $\T:\alg \to \alg$ such that 
\begin{equation}\label{eq:inv}
a^{\T \T} = a,\, (ab)^\T = b^\T a^\T\, \text{and} ~~\alg_{ij}^\T = \alg_{ji}, \quad \text{ for all }  i,j \in \{1,\ldots, r\}.
\end{equation}
With that,  we have $(a^\T)_{ij} = {a_{ji}^{\T}}$,
where ``$a^\T_{ji}$'' should be read as $(a_{ji})^*$. 

Then, a \emph{T-algebra of rank $r$} is a matrix algebra of rank $r$ with involution such that the following additional axioms are satisfied.
\begin{enumerate}[$({a}1)$]
	\item For each $i$,  $\alg _{ii}$ is a subalgebra isomorphic to $\Re$. \label{ax:1}
\end{enumerate}
Let $\rho _i: \alg_{ii} \to \Re$ denote the algebra isomorphism and 
let $\id _i$ denote the unit element in $\alg _{ii}$, \textit{i.e.}, 
the element satisfying $\rho _i(\id _i) = 1$.
Furthermore, define the function $\tr: \alg \to \Re$ by $\tr(a) := \sum _{i=1}^r \rho _i (a_{ii})$.
\begin{enumerate}[$({a}1)$]
	\setcounter{enumi}{1}
	\item For all $a \in \alg$  and all $i,j \in \{1, \ldots, r\}$ we have 
	$\id _i a_{ij} = a_{ij}$ and $a_{ji}\id _i = a_{ji}.$ \label{ax:2}
	\item For all $a,b \in \alg$, $\tr(ab) = \tr(ba)$. \label{ax:3}
	\item For all $a,b,c \in \alg$ and we have $\tr((ab)c) = \tr(a(bc))$. \label{ax:4}
	\item For all $a \in \alg$, we have $\tr(aa^*) \geq 0$ with equality if and only if $a = 0$. \label{ax:5}
	\item For all $a,b,c \in \alg$ and $1 \leq i\leq j \leq k \leq l \leq r$, we have $a_{ij}(b_{jk}c_{kl}) = (a_{ij}b_{jk})c_{kl}$. \label{ax:6}
	\item For all $a,b \in \alg$, $1\leq i\leq j \leq k \leq r$ and $1 \leq l\leq k \leq r$,  we have $a_{ij}(b_{jk}b_{lk}^*) = (a_{ij}b_{jk})b_{lk}^*$. \label{ax:7}
\end{enumerate}
%
We note that these axioms are sometimes stated in different but equivalent forms, see \cite[Remarks~2-7]{CH09} for some equivalences.

Defining $e \coloneqq e_1 + \cdots + e_r$, 
 Axiom~\ref{ax:2} implies that $ea = a$ and $ae = a$ holds for all $a \in \alg$, so $e$ plays the role of identity element.
In addition, for each $a \in \alg$ and $i\in\{1,\ldots,r\}$, $ae_i$ and $e_ia$ correspond to the $i$-th ``column'' and ``row'' of $a$, respectively.

Given a T-algebra $\alg$, the bilinear function
\begin{equation}\label{eq:innprod}
\inProd{a}{b}\coloneqq \tr(ab^*) = \sum_{1\leq i,j\leq r} \rho_i(a_{ij}b_{ij}^*)
\end{equation}
is an inner product on $\alg$ due to Axiom \ref{ax:5}, which induces the norm $\|a\|\coloneqq \sqrt{\inProd{a}{a}}$ over $\alg$.
The subspaces $\alg_{ij}$ are orthogonal to each other under this inner product by \eqref{eq:m_rules} and Axiom \ref{ax:3}.

\begin{remark}[Analogy with the usual matrices]\label{rem:usual}
The space $M^{n\times n}$ of real $n\times n$ matrices can be seen as a T-algebra in a natural way, however, an important 
difference is that for $a \in M^{n\times n}$, $a_{ij}$ is not a scalar, but a matrix as well. For example, for $a \coloneqq \begin{pmatrix}	1 & 2 \\3 & 4\end{pmatrix}$, we  have the following decomposition
\[
a = a_{11} + a_{12} + a_{21} + a_{22} = \begin{pmatrix}	1 & 0 \\0 & 0\end{pmatrix} + \begin{pmatrix}	0 & 2 \\0 & 0\end{pmatrix} +\begin{pmatrix}	0 & 0 \\3 & 0\end{pmatrix} +\begin{pmatrix}	0 & 0 \\0 & 4\end{pmatrix}.  
\]
\end{remark}

Now, we define certain subsets of $\alg$. The subspace of 
``Hermitian'' matrices of $\alg$ is given 
by
\[
\HE \coloneqq \{a \in \alg \mid a^\T = a \}.
\]
We also define sets of upper triangular submatrices, so that
\begin{equation*}
\begin{array}{l}
\displaystyle \UT \coloneqq \bigoplus _{1\leq i\leq j\leq r} \alg _{ij},\\
\UT_+ \coloneqq \{a \in \UT \mid \rho_i(a_{ii}) \geq 0~\text{ if }~ 1 \leq  i\leq r \}, \\
\UT_{++} \coloneqq \{a \in \UT \mid \rho_i(a_{ii}) > 0~\text{ if }~ 1 \leq i  \leq r \}.
\end{array}
\end{equation*}
When it is necessary to emphasize the underlying algebra, we will alternatively write $\HE(\alg)$, $\UT(\alg)$, $\UT_{+}(\alg)$ or $\UT_{++}(\alg)$. 
Finally, the homogeneous cone in $\HE$ associated to the T-algebra $\alg$ is defined as 
\[
\stdHC \coloneqq \{tt^* \mid t \in \UT_{++} \}.
\]
and its closure is given by 
\[
\closure \stdHC = \{tt^* \mid t \in \UT_{+} \}.
\]
Each element in $\stdHC$ has a unique representation  as $tt^\T$ for $t \in \UT_{++}$, see \cite[Proposition~4]{CH09} or \cite[Chapter~III, Proposition~2]{V63}. 
This is analogous to the fact that a positive definite symmetric matrix has a unique Cholesky factorization in terms of upper triangular matrices.
This uniqueness breaks down for elements of $\closure \stdHC$ and, in general, $tt^\T = uu^\T$ for $t,u \in \UT_{+}$ does not imply $u = t$.
Nevertheless, we will discuss how to recover uniqueness 
 in Section~\ref{sec:chol}.

An important result by Vinberg establishes the correspondence between 
homogeneous cones and T-algebras as follows.
\begin{theorem}[{\cite[Chapter~III, \S2 and Theorem~4]{V63}}]\label{theo:vin}
	$\stdHC$ is an open homogeneous convex cone in $\HE(\alg)$. Conversely, let $\stdCone$ be a pointed open homogeneous convex cone contained in a real finite dimensional space. Then, there is a (up to isomorphism) unique 
T-algebra $\alg$ for which $\stdHC$ is linearly isomorphic to $\stdCone$.
\end{theorem}

Let $\stdCone \subseteq \ambSpace$ be a pointed homogeneous convex cone. By restricting $\ambSpace$ to $\spanVec\stdCone$, we may assume that $\stdCone$ has non-empty interior. Then, we define the \emph{homogeneous cone rank} of $\stdCone$ as the rank of the corresponding T-algebra given in Theorem~\ref{theo:vin} for the interior of $\stdCone$. 
This is well-defined because isomorphic T-algebras must have the same rank, see \cite[Definition~5]{V63}.
For a non-pointed  homogeneous closed convex cone $\stdCone$, we may similarly define rank by looking at the rank 
of its pointed component $\stdCone \cap L^\perp$, where $L \coloneqq \stdCone \cap - \stdCone$.

It follows from Vinberg's discussion in \cite{V63} that the rank of a homogeneous cone is invariant under linear isomorphisms. 
We note this as a proposition. 
\begin{proposition}\label{prop:rank}
The homogeneous cone rank is invariant under linear isomorphisms.
\end{proposition}
%
We also note in passing 
that G\"uler and Tun\c{c}el showed that for a homogeneous pointed closed convex cone $\stdCone$ with non-empty interior, the smallest barrier parameter $\vartheta(\stdCone)$ among all logarithmically homogenous self-concordant barriers of $\stdCone$ satisfies $\vartheta(\stdCone) = r$, see \cite[Theorem~4.1]{GT98} and the discussion in \cite[Section~3.1]{CH09}.
Now, $\vartheta(\stdCone)$ is also invariant under linear isomorphisms which can be either verified directly or inferred from, say,  \cite[Theorem~5.3.3]{Ne18}.

\paragraph{Dual cone}
Let $\algD$ be the matrix algebra
\begin{equation}\label{eq:algD}
\algD\coloneqq \bigoplus_{i,j=1}^r \algD_{ij},\quad  \text{where}~~ \algD_{ij}\coloneqq\alg_{r+1-i,r+1-j},
\end{equation}
endowed with the same involution as $\alg$, which forms a T-algebra of rank $r$. Then the following identity holds \cite[p.\,390]{V63}.
\begin{equation}\label{eq:dual_cone}
\stdHC^* = \closure \stdCone(\algD) = \{t^*t \mid t \in \UT_+(\alg)\}.
\end{equation}

\paragraph{Automorphisms}
We now describe how $\UT_{++}$ can be seen as a group acting over $\stdHC$, see \cite[Section~2.2]{CH09} or \cite[Chapter~III, \S 2]{V63} for more details. 
Let $u \in \UT_{++}$ and for $x \in \stdHC$ denote by $t_x$ the unique element in $\UT_{++}$ such that $x = t_x t_x^{\T}$ holds.
We define the map $\tilde Q_u : \stdHC \to \stdHC$ satisfying
\begin{equation}\label{eq:aut}
\tilde Q_u(x) \coloneqq (u t_x)(ut_x)^\T, \qquad \forall x \in \stdHC.
\end{equation}
 Then, $\tilde Q_u$ is a bijection over $\stdHC$ and it turns out  that
\begin{equation}\label{eq:tildequ}
\tilde Q_u(x) = u((t_x t_x^*)u^*) + u(u(t_xt_x^*)) - (uu)(t_xt_x^*)  =u(xu^*) + u(ux) - (uu)(x) 
\end{equation}
holds, see \cite[Proposition~2]{CH09}.
In particular, the map $\tilde Q_u$ is linear  over $\stdHC$.

Next, for $a \in \alg$, define $a_H \coloneqq a+a^*$.
Taking \eqref{eq:tildequ} as a starting point, we define 
the \emph{quadratic map $Q_a : \H(\alg) \to \H(\alg)$} for $a \in \alg$ as the linear map satisfying
\begin{equation}\label{eq:aut2}
Q_a(b) \coloneqq \frac{1}{2}(a(ba^*) + a(ab) - (aa)b)_H,
\end{equation}
e.g., see \cite[Definition~8]{CH09}.
In view of \eqref{eq:tildequ},  $\tilde Q_u(x) = Q_{u}(x)$ holds for $x \in \stdHC$ and $Q_{u}$ can be seen as a linear extension of $\tilde Q_u$.

A useful property is that for $a \in \alg$, the adjoint map $Q_a^*$ satisfies
\begin{equation}\label{eq:qadj}
Q_{a}^\T = Q_{a^\T},
\end{equation}
e.g., see \cite[Section~2.3]{CH09}.


The linear maps $Q_{u}$ for $u \in \UT_{++}$ form a subgroup of the automorphism group of $\stdHC$ acting transitively on $\stdHC$. 
In fact, for any $u \in \UT_{++}$, there exists a unique $u^{-1} \in \UT_{++}$ such that 
that $uu^{-1} = u^{-1}u = e$ holds, see \cite[Proposition~1]{CH09}.
More generally, we have 
\begin{align}
Q_{u} Q_t & = Q_{ut}, \qquad \forall u,t \in \UT_{+}, \label{eq:qut}\\
Q_{u}^{-1} & = Q_{u^{-1}}, \qquad \forall u \in \UT_{++}, \label{eq:qinv}\\
Q_{u}(tt^*) & = (ut)(ut)^*, \qquad \forall u,t \in \UT_{+} \label{eq:qut2},
\end{align}
where \eqref{eq:qut} follows from \cite[Corollary~1]{CH09} by taking limits,
\eqref{eq:qinv} follows from \cite[Proposition~1 and Corollary~1]{CH09} and 
\eqref{eq:qut2} can be obtained by taking limits in \eqref{eq:aut} or by invoking \cite[Proposition~2]{CH09} and \eqref{eq:aut2}.
%

In what follows, we will say that an automorphism $Q$ of $\stdHC$ is \emph{triangular} if there 
exists $u \in \UT_{++}$ such that $Q = Q_u$ holds.




\paragraph{Principal subalgebras}
Let $I \subseteq \{1,\ldots, r\}$ be a subset of indices. We will identify certain subalgebras of $\alg$ that correspond to ``principal submatrices''. 
We define $\alg_I$ to be the following subalgebra
\[
\alg_I \coloneqq \{a \in \alg \mid a_{ik} = a_{ki} = 0, \forall i \in I, \forall k \in \{1,\ldots,r\} \}.
\]
For example, $a \in \alg _{\{r\}}$ if and only if the $r$-th ``column'' and the $r$-th ``row'' of $a$ vanishes. 
We will say that $\alg_{I}$ is a \emph{principal subalgebra} of $\alg$.

$\alg_{I}$ can be seen as a T-algebra of rank $s \coloneqq r - |I|$ in a natural way using the structure inherited from $\alg$, where $|I|$ is the number of elements of $I$. 
Let $J \coloneqq  \{1,\ldots, r\} \setminus I$ be the remaining indices ordered as \[
n_1 < n_2 < \cdots < n_{s}.
\]
This implies that $i < j \Leftrightarrow n_i < n_j$, for all $i,j \in \{1,\ldots, s\}$.
In essence, we are ``renaming'' the remaining indices in $J$ to indices in the set $\{1,\ldots, s\}$, in such a way that their ordering in $J$ is preserved. 
This will be important to ensure compatibility with the T-algebra structure of $\alg$.

The bigradation of $A_{I}$ inherited from $\alg$ is given by 
$\alg_{I} = \bigoplus _{i,j=1}^s \bar{\alg}_{ij}$, where 
 $\bar{\alg}_{ij} \coloneqq \alg _{n_{i}n_{j}}$.
For $i,j,k,l \in \{1,\ldots, s\}$, we have
\[
\bar \alg _{ij} \bar \alg _{lk} = \alg _{n_{i}n_{j}} \alg _{n_{k}n_{l}}.
\]
Since $n_j = n_k$ holds if and only if $j = k$, we have
\begin{equation*}\begin{aligned}
\bar\alg _{ij} \bar\alg _{jk} & \subseteq \bar\alg _{ik} \\
\bar\alg _{ij} \bar\alg _{kl} & = \{0\} \qquad \text{if } j \neq k.  
\end{aligned} 
\end{equation*}
Defining $\bar\rho _i \coloneqq \rho _{n_i}$, $\bar{e}_{i} \coloneqq e_{n_{i}}$ and restricting the trace and involution functions to $\alg_{I}$, it is straightforward (albeit tedious) to check that 
Axioms~\ref{ax:1} through \ref{ax:7} still hold. For example,  the restriction of the involution $\T$ is still an involutive  bijection from $\alg_I$ to $\alg_I$, because if $a_{ik} = a_{ki} = 0$, then $a_{ik}^\T = a_{ki}^*  = 0$ must also hold. Furthermore, we have
\[
\bar{\alg}_{ij}^\T = {\alg}_{n_{i}n_{j}}^\T = {\alg}_{n_{j}n_{i}} = \bar{\alg}_{ji}
\] 
For the rest of the axioms the idea is that since they are true over $\alg$ for all the possible indices $\{1,\ldots, r\}$ they must be true for any subset $J \subset \{1,\ldots, r\}$ of those indices. This takes 
care of the validity of axioms~\ref{ax:1} to \ref{ax:5} over $\alg_{I}$.
Next, if  $1 \leq i\leq j \leq k \leq l \leq s$, then we must 
have $1 \leq n_i\leq n_j \leq n_k \leq n_l \leq s$, so Axiom~\ref{ax:6} is also valid for $\alg_{I}$. A similar argument holds for Axiom~\ref{ax:7}.

Finally, we take a look at the automorphisms of $\stdCone(\alg_I)$. As before, 
the triangular matrices $\UT_{++}(\alg _I)$ induce a group that acts transitively on $\stdCone(\alg_I)$ and in the next result we will see that they can be seen as restrictions of automorphisms of $\stdHC$. 

\begin{proposition}\label{prop:sub_aut}
Given a triangular automorphism $Q_u$ of $\stdCone(\alg_I)$, there exists 
a triangular automorphism $Q_{\tilde u}$ of $\stdCone(\alg)$ such that 
the restriction of $Q_{\tilde u}$ to $\H(\alg_I)$ coincides with 
$Q_u$.
\end{proposition}
\begin{proof}
 We ``complete''  $u$ by adding $1$'s to the diagonal by letting $\tilde u  = u + \sum_{i \in I} e_i$. In this way, 
 $\tilde u \in \UT_{++}(\alg)$.
Next, let $t \in \UT_{++}(\alg_I)$. 
By \eqref{eq:qut2} we have $Q_{\tilde u}(tt^*) = (\tilde u t)(\tilde u t)^*$.
As a consequence of the multiplication rules in \eqref{eq:m_mult} and Axiom~\ref{ax:2} we have
\[
\tilde u t = \left(u + \sum _{i \in I} e_i\right)t = ut + \sum _{i \in I}\sum _{j = 1}^r t_{ij}.
\]
However, since $t \in\UT_{++}(\alg_I)$, we have $t_{ij} = 0$ whenever $i \in I$, so the second summation is zero and we have $\tilde u t = u t$.
Therefore
\[
Q_{\tilde u}(tt^*) = (ut)(ut)^* = Q_{u}(tt^*)
\]
holds which implies that $Q_{\tilde u}$ and $Q_{u}$ coincide over $\stdCone(\alg_I) = \{tt^* \mid t \in \UT_{++}(\alg_I)\}$ and must coincide over $\H(\alg_I) = \spanVec \stdCone(\alg_I)$.
%
%
\end{proof}


\section{Facial structure of homogeneous cones}\label{sec:faces}
In this section, we prove our main results on the facial structure of general homogeneous cones. 
Before that, we must discuss the generalized Cholesky decomposition in homogeneous cones. 

\subsection{Unique Cholesky factorization in homogeneous cones}\label{sec:chol}
Every element in the interior of a homogeneous cone has a unique representation as $tt^*$, with $t \in \UT_{++}$. 
Then, a limiting argument tells us that elements in the boundary can be represented as $tt^*$ with $t \in \UT_{+}$, however this factorization is not unique in general. 
We will address this issue in this subsection.

We say that $t \in \UT_+$ is \emph{proper} if $t_{ii} = 0$ implies that 
the $i$-th ``column'' of $t$ vanishes, i.e., $t_{ki} = 0$ for all $k < i$.
It turns out that each $x \in \closure \stdHC$ has a unique decomposition in terms of \emph{proper} triangular matrices. 
In order to prove that, first we need the following  lemma.

\begin{lemma}\label{lemma:row_vanish}
	Let $\alg$ be a T-algebra of rank $r$. 
	The proper face $(\closure \stdHC )\cap \{e_r\}^\perp$ of $\closure\stdHC$ is equal to $\closure (\stdCone(\alg_{\{r\}}))$.
\end{lemma}
\begin{proof}
Let $x \in 	(\closure \stdHC )\cap \{e_r\}^\perp$.
Then we have $x_{rr} = 0$, which implies $x_{ir} = x_{ri} = 0$ for all $i \in \{1,\ldots r\}$, by  \cite[Proposition~2.4]{KTX12}.

Since $e \in \stdHC$, for every 
$\epsilon > 0$, $x+ \epsilon e  \in \stdHC$ and, therefore, $x+ \epsilon e $ has a unique Cholesky decomposition in upper triangular matrix with positive diagonal. That is,  $x+\epsilon e = (t^{\epsilon})(t^{\epsilon})^\T$ 
for $t^{\epsilon} \in \UT_{++}$.
For the last column of $x$ we have
\begin{equation}\label{eq:last_column_lem}
(x+\epsilon e)_{ir} = \sum _{k=1}^{r}  t^{\epsilon}_{ik}((t^{\epsilon})^\T)_{kr} =  t^\epsilon_{ir} (t^\epsilon)^\T_{rr} = \rho_{r}(t_{rr}^{\epsilon})t_{ir}^{\epsilon}, \quad \forall i \in \{1,\ldots, r\}.
\end{equation}
Because $(x+\epsilon e)_{ir} = x_{ir}=0$ for $i \neq r$ and $\rho_{r}(t_{rr}^{\epsilon}) > 0$, \eqref{eq:last_column_lem} leads to
\begin{equation*}
t_{ir}^{\epsilon} = 0, \qquad \forall i \neq r.
\end{equation*}
We also obtain from \eqref{eq:last_column_lem} and $(x+\epsilon e)_{rr} = \epsilon e_{r} $ that $\rho_{r}(t_{rr}^\epsilon) = \sqrt{\epsilon}$.

Gathering all the facts and observing that $\norm{t^{\epsilon}}^2=\tr((t^\epsilon) (t^\epsilon)^*)=\tr(x+\epsilon e)$ is bounded for bounded $\epsilon$, 
there exists $\epsilon_k \downarrow 0$ for which the sequence $\{t^{\epsilon_k}\}$ converges to a point $\hat t$ of $\UT_{+}$ that satisfies 
$x = \hat t \hat t ^\T$ and such that the last column of $\hat t$ vanishes. 
In particular, $\hat t$ belongs to the set of upper triangular matrices of the principal subalgebra $\alg_{\{r\}}$ and thus $x \in \closure (\stdCone(\alg_{\{r\}}))$ holds.
For the converse, if $ x \in \closure (\stdCone(\alg_{\{r\}}))$, then 
$0 = \rho_r(x_{rr}) = \inProd{x}{e_r}$, so that $x \in (\closure \stdHC )\cap \{e_r\}^\perp$.
\end{proof}


The following proposition was essentially proven by Gindikin \cite[Chapter~2, \S1, Lemma~7]{Gin92}. 
Chua also provides a similar result in \cite[Lemma~17]{Chua06} by first showing the analogous result for positive semidefinite matrices and making use of matrix realizations of homogeneous cones.
Here, we provide a different proof from which we will be able to extract a factorization algorithm.
\begin{proposition}[Unique Cholesky factorization]\label{prop:chol}
	Let $x \in \closure \stdHC$, then $x = tt^*$ for a unique proper $t \in \UT_+$.
\end{proposition}
\begin{proof}
	We proceed by induction on the rank $r$ of the  T-algebra.
	The proposition is true if $r = 1$. So let us assume that it holds for all T-algebras of rank $(r-1) \geq 1$ and let us show that it holds for T-algebras of rank $r$, for $r \geq 2$.

	Let $\alg$ be a T-algebra of rank $r$ and let $x \in \closure \stdHC$.
	Let $t \in \UT_{+}$ be such that $x = tt^\T$. 
	For the last column of $x$ we have
	\begin{equation}\label{eq:last_column}
	x_{ir} = \sum _{k=1}^{r}  t_{ik}(t^\T)_{kr} =  t_{ir} t^\T_{rr} = \rho_{r}(t_{rr})t_{ir}, \quad \forall i \in \{1,\ldots, r\}.
	\end{equation}
	In particular, $x_{rr} = \rho_{r}(t_{rr})t_{rr} = t_{rr}^2$ holds. 
	We then consider two cases.
	
	If $\rho _{r}(x_{rr}) = 0$, i.e., $\inProd{x}{e_{r}} = 0$, then, by Lemma~\ref{lemma:row_vanish}, 
	$x$ belongs to the closure of the homogeneous cone associated to the principal subalgebra $\alg_{\{r\}}$, i.e., $x \in  \closure (\stdCone(\alg_{\{r\}}))$. 
	By the induction hypothesis, $x = \hat t \hat t^\T$ for a unique proper upper triangular matrix $\hat t$ of  $\alg_{\{r\}}$.
	Since $\hat t$ is also  proper with respect to $\alg$ (after all, the last column of $\hat t$ vanishes), $\hat t$ is the desired proper element satisfying $x = \hat t \hat t^\T$.
	We note that $\hat t$ is unique over $\UT_+$ because  \eqref{eq:last_column} implies that all $t$ satisfying $x = tt^*$ must have  $t_{rr} = 0$, which for any proper $t$ leads to $t_{ir} = 0$ for all $i$. This proves the assertion in the case $\rho _{r}(x_{rr}) = 0$.
	
	Next, suppose that $\rho_{r}(x_{rr})>0$. Define
	$$u \coloneqq \sum_{i=1}^r x_{ir}/\sqrt{\rho_r(x_{rr})},\qquad \hat{x} \coloneqq x - uu^*.$$
	We claim that $\hat{x} \in  \closure (\stdCone(\alg_{\{r\}}))$.
	In fact, we can write $x = tt^*$ for some $t\in \UT_+$ because $x \in  \closure (\stdCone(\alg))$. For any such $t$,
	\eqref{eq:last_column} implies that
	\begin{equation}\label{eq:last_column_eq}
	t_{ir} = x_{ir}/\sqrt{\rho_{r}(x_{rr})} = u_{ir},\quad i=1,\ldots,r.
	\end{equation}
	This means that the last column of $t$ is uniquely determined by $u$, i.e., $te_r=u$ holds.
	Therefore, we have $tu^* = (te_1+\cdots+te_r)(te_r)^* = (te_r)(te_r)^* = uu^*$ and 
	\begin{equation}\label{eq:hatx_decomp}
	(t-u)(t-u)^* =  tt^*-uu^* = x - uu^* = \hat{x}.
	\end{equation}
	Since $t-u \in \UT_+(\alg_{\{r\}})$ holds by \eqref{eq:last_column_eq}, the identity \eqref{eq:hatx_decomp} implies the claim $\hat{x} \in \closure (\stdCone(\alg_{\{r\}}))$.
	By the induction hypothesis, there is a unique proper $t_{\hat{x}} \in \UT_+(\alg_{\{r\}})$ satisfying $\hat{x}=t_{\hat{x}}t_{\hat{x}}^*$. Defining $t_{x}\coloneqq t_{\hat{x}}+u \in \UT_+$, we see that $t_{x}$ is proper and satisfies
	$$
	t_xt_x^* = (t_{\hat{x}}+u)(t_{\hat{x}}+u)^*=t_{\hat{x}}t_{\hat{x}}^* + uu^* = \hat{x}+uu^*=x.
	$$
	It remains to show the uniqueness of $t_x$. Suppose $x=tt^*$ for some proper $t\in \UT_+$. As discussed in \eqref{eq:last_column_eq} and \eqref{eq:hatx_decomp}, the last column of $t$ is given by $u$ and $\hat{x}=(t-u)(t-u)^*=t_{\hat{x}}t_{\hat{x}}^*$ holds. In addition, since $t$ is proper, $t-u$ is a proper element in $\UT_+(\alg_{\{r\}})$, which implies $t-u=t_{\hat{x}}$ by the uniqueness of $t_{\hat{x}}$. Hence, we obtain $t = t_{\hat{x}}+u = t_x$, i.e., the uniqueness of $t_x$ follows.	
\end{proof}
The proof of Proposition~\ref{prop:chol} shows that if $x \in \closure \stdHC$ and $x = tt^\T$ holds for a proper $t \in \UT_{+}$ then, the last 
column of $t$ is either zero (if $x_{rr} = 0$) or is uniquely determined by \eqref{eq:last_column_eq}. Denoting such a column by $u$ and the remaining $r-1$ columns of $t$ by $\hat t$, \eqref{eq:hatx_decomp} implies that $\hat t$ satisfies
\[
\hat t \hat t^\T = x - uu^\T.
\]
Now, $\hat t$ is an element of $t \in \UT_{+}(\alg_{\{r\}})$, so again, the last column of $\hat t$ in $\alg_{\{r\}}$ (i.e., the $(r-1)$-th column of $\hat t$ seen as an element of $\alg$), is uniquely determined and is either zero (if $(x - uu^\T)_{r-1,r-1} = 0$) or is given by \eqref{eq:last_column_eq} with $x-uu^\T$ in place of $x$.
This leads to an algorithm for computing a generalized Cholesky decomposition, see Algorithm~\ref{alg:cholesky}.
Note that Algorithm~\ref{alg:cholesky} can also be used to test membership in 
$\stdHC$: given some $x \in \HE(\alg)$, we invoke Algorithm~\ref{alg:cholesky} with $x$ as input. Denoting by $t$ the output of Algorithm~\ref{alg:cholesky}, 
we have $x \in \stdHC$ if and only if $x = tt^*$. 

%

\begin{algorithm}[h]
	\caption{Generalized Cholesky decomposition in a homogeneous cone }
	\label{alg:cholesky}
	\begin{algorithmic}[1]
		\REQUIRE{$x \in \closure\stdHC$}
		\STATE $u \leftarrow 0$, $t \leftarrow 0$, $y \leftarrow x $
		
		\FOR{$i = r,\ldots,1$}
		\STATE $y \leftarrow y - uu^\T$
		\IF {$\rho_{i}(y_{ii}) > 0$}
		\STATE $t_{ji} \leftarrow y_{ji}/\sqrt{\rho_{i}(y_{ii})}, \quad \forall j \in \{1,\ldots, i\}.$
		\ELSE
		\STATE $t_{ji} \leftarrow 0, \quad \forall j \in \{1,\ldots, i\}$
		\ENDIF
		\STATE $u \leftarrow$ $i$-th column of $t$
		\ENDFOR
		\RETURN{$t$}
	\end{algorithmic}
\end{algorithm}

\begin{proposition}\label{prop:row_vanish}
	Let $\alg$ be a T-algebra of rank $r$. Each 
	$e_i$ generates extreme rays of both $\closure \stdHC$ and $\stdHC^*$.
	The proper face $(\closure \stdHC )\cap \{e_i\}^\perp$ of $\closure\stdHC$ is equal to $\closure (\stdCone(\alg_{\{i\}}))$.
\end{proposition}
\begin{proof}
The inclusion $\closure (\stdCone(\alg_{\{i\}})) \subseteq (\closure \stdHC )\cap \{e_i\}^\perp$ is immediate, so we focus on the converse. 
Let $x \in (\closure \stdHC )\cap \{e_i\}^\perp$ and let $x = tt^*$ for some 
proper $t \in \UT_{+}$ as in Proposition~\ref{prop:chol}. Then, $0=\inProd{x}{e_i} = \rho_i(x_{ii}) $ implies that 
$x_{ii} = 0$. In view of $x = tt^\T$, this leads to
\[
0 = \sum _{k=i}^r t_{ik}(t_{ik})^\T.
\]
Therefore, the $i$-th row of $t$ vanishes and since $t$ is proper, we have that the $i$-th column of $t$ also vanishes. 
We conclude that $t \in \UT_{+}(\alg_{\{i\}})$ and, therefore, 
$x \in \closure (\stdCone(\alg_{\{i\}}))$.


	Next, we verify that $e_i$ generates an extreme ray of $\closure \stdHC$. Suppose that $x + y = e_{i}$ holds for $x,y \in \closure \stdHC$. Then, $\rho_{k}(x_{kk}) + \rho_{k}(y_{kk}) = 0$ for $k \neq i$. Since $\rho_{k}(x_{kk})$ and $\rho_{k}(y_{kk})$ are nonnegative, we have 
	$\rho_{k}(x_{kk}) = \rho_{k}(y_{kk}) = 0 = \inProd{x}{e_k} = \inProd{y}{e_{k}}$. From \cite[Proposition~2.4]{KTX12}, this implies that $0 = x_{lk} = x_{kl} = y_{lk} = y_{kl}$, for all $k \neq i$ and all $l \in \{1,\ldots,r\}$.
	In particular, only $x_{ii}$ and $y_{ii}$ can be nonzero and therefore, they must be in the half-line generated by $e_i$.
	
	We recall that the dual cone is the cone corresponding to the T-algebra $\algD$ defined in \eqref{eq:algD} and each $e_i$ is the identity element of $\algD_{r+1-i,r+1-i}$. Applying what we proved so far to $\closure \stdCone(\algD) = \stdHC^*$ shows that 
	the $e_i$ are also extreme rays of $\stdHC^*$.
\end{proof}


\subsection{Faces and their orbits under the action of triangular matrices}
In this subsection, we present our main results on the facial structure of homogeneous cones. 
We start by observing that the cones induced by principal subalgebras are faces that are orthogonally projectionally exposed.
\begin{lemma}\label{lem:orth_proj}
Let $\alg$ be a T-algebra of rank $r$ and let $I \subseteq \{1,\ldots, r\}$ be nonempty. Then, $\closure (\stdCone(\alg_I)) $ is an orthogonally projectionally exposed face of $\closure \stdHC$.	In particular, 
$Q_{e_I}(\closure \stdHC) = \closure (\stdCone(\alg_I))$ holds, where 
$e_I \coloneqq \sum _{i \not \in I} e_i$.
\end{lemma}
\begin{proof}
From Proposition~\ref{prop:row_vanish} we have  
that each $\closure \stdCone(\alg_{\{i\}})$ for $i \in \{1, \ldots, r\}$ is a face and we have
\[
\closure (\stdCone(\alg_I)) = \bigcap _{i \in I} \closure (\stdCone(\alg_{\{i\}})) = 
\bigcap _{i \in I}( \closure \stdHC \cap \{e_i\}^\perp).
\]
Since an intersection of faces is a face, this shows that 
$\closure (\stdCone(\alg_I)) $ is a face of $\closure \stdHC$.

Now, let $e_I$ be the identity element of $\alg_I$, i.e., 
$e_I = \sum _{i \not \in I} e_i$. We will consider the quadratic map 
$Q_{e_I}$ and show that $Q_{e_I}$ is the desired projection.

First, we observe that $Q_{e_I}$ is idempotent,
since $Q_{e_I}Q_{e_I} = Q_{e_I e_I} = Q_{e_I}$ holds, where 
the first equality follows from \eqref{eq:qut} and the second from \eqref{eq:m_rules} and Axiom~\ref{ax:2}. The map $Q_{e_I}$ is also self-adjoint by \eqref{eq:qadj} since $e_I^\T = e_I$. This shows that $
Q_{e_I}$ is an orthogonal projection.

Next, let $x \in \stdHC$ and write $x = tt^\T$ for some $t \in \UT_{++}$, then, in view of \eqref{eq:qut2} we have
\[
Q_{e_I}(x) = (e_It)(e_It)^*.
\]
Denoting the ``rows'' of $t$ by $t^i$, so that $t = t^1 + \cdots+ t^r$ holds and letting $t_I \coloneqq \sum _{i \not \in I}t^i$ we have $e_It = t_I$ and
\[
Q_{e_I}(x) = (t_I)(t_I)^*.
\]
Computing the diagonal elements of $Q_{e_I}(x)$ for $i \in I$ we obtain 
\[
(Q_{e_I}(x))_{ii} = \sum _{j=i}^r ({t_{I})}_{ij} ({(t_{I})}_{ij})^* = 0,
\]
since $(t_{I})_{ij} = 0$ for $i \in I$. Then, Proposition~\ref{prop:row_vanish} and $Q_{e_I}(x) \in \closure \stdHC$ implies that 
$Q_{e_I}(x) \in \closure(\stdCone(\alg_{I}))$. 
Therefore, the projection $Q_{e_I}$ maps elements of 
$\closure \stdHC$ to elements of $\closure(\stdCone(\alg_{I}))$.
Finally, since $e_I$ is the identity element in $\alg_I$, we have
$Q_{e_I}(x) = x$ for all $x \in \H(\alg_I)$ by \eqref{eq:aut2}.
\end{proof}
Motivated by Lemma~\ref{lem:orth_proj}, in what follows we will call a face of the format $\closure (\stdCone(\alg_I)) $ a \emph{principal face}, since they are induced by principal subalgebras of $\alg$.

Our next task is showing that every element of $\closure \stdHC$ can be transformed to the identity element of some principal subalgebra via a triangular automorphism.

\begin{lemma}\label{lem:qei}
Let $\alg$ be a T-algebra of rank $r$ and $x \in \closure \stdHC$.
Let $x = tt^\T$ for a proper $t \in \UT_+$ as in Proposition~\ref{prop:chol}.
There exists a triangular automorphism $Q_u$ of $\stdHC$ such that
\[Q_u(x) = e_I,\] where $I \coloneqq \{i \mid \rho_i(t_{ii}) = 0\}$, $e_I \coloneqq \sum _{i \not \in I} e_i$.
\end{lemma}
\begin{proof}
Let $t_{1}, \ldots, t_{r}$ denote the ``columns'' of $t$, so 
that
\[
t = \sum _{i=1}^r t_{i} = \sum _{i \not \in I} t_i,
\]	
where the second equality holds because $t$ is assumed to be proper.

For $i,j \in \{1, \ldots, r\}$ we
have
\[
t_i e_j = \sum_{1\leq k \leq i \leq r} t_{ki}e_j.
\]
By the multiplication rules in \eqref{eq:m_mult} and Axiom~\ref{ax:2}, 
we have
\[
t_i e_j = \begin{cases}
t_i & \text{ if } i = j \\
0 & \text{ otherwise. }
\end{cases}
\]
With this in mind, we define $\tilde u \in \UT_{++}$ as follows: for 
$i \in I$, we set the $i$-th column of $\tilde u$ to be equal to $e_i$ and 
for $i \not \in I$, we set the  $i$-th column of $\tilde u$ to be $t_i$.
We have
\[
\tilde u e_I = \left(\sum_{j \in I} e_j + \sum _{i \not \in I} t_i\right)\left(\sum _{k \not \in I} e_k\right) =\sum _{i \not \in I} t_i = t.
\]	
With that, $Q_{\tilde u}$ maps $e_I = e_I e_I^\T$ to $(\tilde u e_I)(\tilde u e_I)^\T = t t^\T = x$, see \eqref{eq:qut2}.  By \eqref{eq:qinv}, letting $u \coloneqq \tilde u ^{-1}$,
$Q_{u} = Q_{\tilde u}^{-1}$ is a triangular automorphism  mapping $x$ to $e_{I}$.
\end{proof}
Gathering all the pieces collected so far, we have the following result, which contains, in particular, an  analogue of \eqref{eq:psd_faces} and the subsequent discussion.
\begin{theorem}\label{th:hface}
	Let $\stdFace$ be a proper face of a homogeneous cone $\closure \stdHC$ of rank $r$. The following items hold.
	\begin{enumerate}[$(i)$]
		\item $\stdFace$ is projectionally exposed.
		\item $\stdFace$ is a homogeneous cone of rank $s < r$ and		
		 there is a triangular automorphism $Q_u$ of 
		$\stdHC$ such that $Q_u (\stdFace) = \closure \stdCone(\alg_{I})$ holds for some nonempty $I \subseteq \{1,\ldots, r\}$. 
		Furthermore, letting $x \in \reInt \stdFace$ and writing $x = tt^*$ for a proper $t \in \UT_{+}(\alg)$, we have $I = \{i \mid \rho_i(t_{ii}) = 0 \}$ and $s = r - |I|$.
		\item There is a subgroup $\mathcal{G}$ of automorphisms of $\stdFace$ acting simply transitively on the relative interior of $\stdFace$ such that each $g \in \mathcal{G}$ is obtained through the restriction of a triangular automorphism of $\stdHC$ to $\stdFace$.
	\end{enumerate}	
\end{theorem}
\begin{proof}
Let $x \in \reInt \stdFace$ be arbitrary and consider a decomposition of $x = tt^*$ with a proper $t$ as in Proposition~\ref{prop:chol}.
By Lemma~\ref{lem:qei}, there exists a triangular automorphism $Q_u$ that 
maps $x$ to $e_I$, where $I = \{i \mid \rho_i(t_{ii}) = 0 \}$. We have
 \begin{equation}\label{eq:qaux}
Q_u(\stdFace) = Q_u (\minFace(x, \closure\stdHC)) = \minFace(Q_u(x), \closure\stdHC) = \minFace(e_I, \closure\stdHC) = \closure \stdCone(\alg_I),
 \end{equation} 
 where the first equality follows from \eqref{eq:minF}. The second equality holds because $Q_u$ is an automorphism, so it maps a face to another face and is also a bijection  between the relative interiors. 
 Finally, the last equality holds because $e_I \in \reInt (\closure \stdCone(\alg_I) )= \stdCone(\alg_I)$. 
 We note that if $|I| = 0$ we have $e_I = e$ and $\stdCone(\alg_I) = \stdHC$, which  would contradict the fact that $\stdFace$ is proper.
 We have thus obtained item $(ii)$ since homogeneity and rank are preserved by linear isomorphisms (see Proposition~\ref{prop:rank}) and $\stdCone(\alg_I)$ is a homogeneous cone of rank $r - |I| < r$. 

 Next, let $Q_{e_I}$ be the orthogonal projection mapping $\closure \stdHC$ to $\closure\stdCone(\alg_{I})$ as in Lemma~\ref{lem:orth_proj}.
Then, \eqref{eq:qaux} implies that $ Q_{u^{-1}}Q_{e_I} Q_u $ is a  projection 
mapping $\closure \stdHC$ to $\stdFace$, which proves item~$(i)$.

Let $\bar{\mathcal{G}}$ be the group of triangular automorphisms of $\stdCone(\alg_I)$ and we recall that $\bar{\mathcal{G}}$  acts simply transitively on $\stdCone(\alg_I)$. In view of 
\eqref{eq:qaux}, ${\mathcal{G}} \coloneqq Q_{u}^{-1} \bar{\mathcal{G}}Q_{u}$ acts simply transitively on 
$\stdFace$. Also, Proposition~\ref{prop:sub_aut} tells us that for every $Q \in \bar{\mathcal{G}}$, there exists some triangular automorphism $Q_v$ that coincides with $Q$ over $\H(\alg_I)$ which implies that $Q_{u}^{-1} Q Q_{u}$ coincides 
with $Q_{u^{-1}}Q_v Q_{u}$ over $\stdFace$.
We note that $Q_{u^{-1}}Q_v Q_{u}$ is a triangular automorphism of $\stdHC$ since 
$Q_{u^{-1}}Q_v Q_{u} = Q_{u^{-1}vu}$ holds by \eqref{eq:qut} and there is no ambiguity regarding 
$u^{-1}vu \in \UT_{++} $ since $(u^{-1}v)u = u^{-1}(vu)$ which is implied by Axiom~\ref{ax:6}.
Overall, we conclude that, indeed, there exists a subgroup acting simply transitively over $\stdFace$ such that each element arises as a restriction of a triangular automorphism of $\stdHC$. 
This, together with \eqref{eq:qaux}, proves item~$(iii)$. 
\end{proof}
Theorem~\ref{th:hface} provides a relatively complete picture of the facial structure of a homogeneous cone. 
It tells us that not only the faces are projectionally exposed and homogeneous cones themselves, but their homogeneous structure can be obtained from the original cone.

Furthermore, each face is linearly isomorphic to a principal face through an automorphism of the cone. 
Finding the automorphism is an entirely algorithmic process. 
Given some $x \in \reInt \stdFace$, we first decompose $x$ as $tt^\T$  for $t$ a  proper matrix as in Algorithm~\ref{alg:cholesky}, then we compute the index set $I = \{i \mid \rho_i(t_{ii}) = 0 \}$, which reveals the principal face that is isomorphic to $\stdFace$. The actual isomorphism can then be computed by 
finding $u \in \UT_{++}$ such that
\begin{equation}\label{eq:utei}
ut = e_I
\end{equation}
holds and  $u$ always exists by Lemma~\ref{lem:qei}. Solving this system does not pose difficulties given the upper triangular structure of $u$ and $t \in \UT_{+}$.
Then, $Q_u$ is an automorphism of $\closure \stdHC$ mapping $\stdFace$ to a principal face. Letting $\bar e \coloneqq e - e_I$ and noticing that Proposition~\ref{prop:row_vanish} implies that  
$\bar{e}$ exposes the principal face induced by $I$, 
this discussion also leads to the following expression for $\stdFace$
\begin{equation}\label{eq:exp}
\stdFace = \{x \in \closure \stdHC \mid \inProd{Q_u(x)}{\bar{e}} = 0 \} = 
\{x \in \closure \stdHC \mid \inProd{x}{Q_{u^\T}(\bar{e})} = 0 \},
\end{equation}
which also reveals that $Q_{u^\T}(\bar{e})$ is an exposing vector for $\stdFace$.
Following the proof of Theorem~\ref{th:hface}, a projection mapping $\closure \stdHC$ to $\stdFace$ is 
$Q_{v}$, for $v \coloneqq u^{-1}e_Iu$. 
We will see a complete example of this process in Section~\ref{sec:chord}.

Recalling the discussion of the faces of PSD matrices in \eqref{eq:psd_faces}, we can see several similarities but two important difference are that, in general, $Q_{v}$ may not be self-adjoint and $Q_u$ is not necessarily a linear isometry.

Finally, we remark that Chua had previously observed a description of the faces of the PSD cone using triangular matrices, see \cite[Equation~(2.4)]{Chua06}.
Chua also proved the following result for homogeneous cones.
Let $y \in \stdFace$ and define $B$ to be the set of $i$'s for which there exists at least one $z \in \stdFace$ whose unique Cholesky factor $t$ satisfies $t_{ii} \neq 0$.
Writing $y = \hat t \hat t ^\T$ for a proper matrix $\hat t$,  Chua showed that $y \in \reInt \stdFace$ if and only if $\rho_i(\hat t) > 0$ for all $i \in B$, see \cite[Proposition~18]{Chua06}. 
Intuitively this means that, inside $\stdFace$, the elements of the relative interior are precisely the ones that have as many nonzeros as possible in the diagonal of their unique Cholesky factors. 
We note that $B$ itself is not enough decide membership on $\reInt \stdFace$ for an arbitrary $y \in \closure \stdHC$, since there are finitely many possible $B$'s and $\closure \stdHC$ may have infinitely many faces. 
Nevertheless, under the setting of Theorem~\ref{th:hface}, we have $I = \{1,\ldots, r\} \setminus B$, so $B$ identifies the principal face of $\closure \stdHC$ that is linearly isomorphic to $\stdFace$. 
Although the statement \cite[Proposition~18]{Chua06} is only about membership in the relative interior of $\stdFace$ for some $y \in \stdFace$,  its actual proof hints at an expression for $\stdFace$ that is somewhat related to item~$(ii)$ of Theorem~\ref{th:hface}.

Before we move on, we present an example.
\begin{example}
	We consider the \emph{matrix norm cone}, which is a spectrahedral cone used in rank minimization heuristics, see \cite[Section~5.1]{TV23}. 
	It was analyzed in slightly different form in \cite[Example~5]{GT98}. We  define it as
	$$\mathcal{M}_{m,n} \coloneqq \left\{ x \in \psdcone{m+n} \mid  x=\begin{pmatrix}
	V & U \\ U^\T & \alpha I_n
	\end{pmatrix},\alpha \in \Re, U \in \Re^{m\times n}, V \in \Re^{m\times m}\right\},$$
	where $I_n$ is the $n\times n$ identity matrix and $U^\T$ is the usual matrix tranpose of $U$.
	
	We consider the following algebra of matrices
	\begin{equation}\label{eq:matrix_norm}
	\alg^{m,n} \coloneqq \left\{ x \in \Re^{(m+n)\times (m+n)} \mid x = \begin{pmatrix}
	 V & W \\ U^\T & \alpha I_n
	\end{pmatrix}, \alpha \in \Re, W \in \Re^{m\times n}, U \in \Re^{m\times n}, V \in \Re^{m\times m}  \right\}.
	\end{equation}
The multiplication is given by 
	\[
	\begin{pmatrix}
	V_1 & W_1 \\ U_1^\T & \alpha_1 I_n
	\end{pmatrix} 	\begin{pmatrix}
	V_2 & W_2 \\ U_2^\T & \alpha_2 I_n 
	\end{pmatrix}  = 
	\begin{pmatrix}
	V_1V_2+W_1U_2^\T & V_1W_2+\alpha_2W_1 \\ U_1^\T V_2+\alpha_1U_2^\T & (\tr(U_1 W_2^\T)+\alpha_1\alpha_2)I_n 
	\end{pmatrix},
	\]	
	where $\tr(U_1 W_2^\T)$ is the usual matrix trace of $U_1 W_2^\T$.
	The multiplication in $\alg^{m,n}$ differs from the usual matrix multiplication only in the lower right block.
	
	It turns out that  $\alg^{m,n}$ is a $T$-algebra of rank $m+1$. 
	Let $x$ be as in \eqref{eq:matrix_norm} and we will describe the bigradation of $\alg^{m,n}$ in terms of the components of $x$. $\alg^{m,n}_{m+1,m+1}$ is isomorphic to $\RR$ and corresponds to $\alpha$, i.e.,  $\alg^{m,n}_{m+1,m+1} \coloneqq \left \{
	\begin{psmallmatrix}
	0 & 0 \\ 0 & \alpha I_n
	\end{psmallmatrix} \mid \alpha \in \RR \right\}$. 
		For $i < m+1$, $\alg^{m,n}_{i,m+1}$ is isomorphic to $\RR^n$  and it corresponds to the $i$-th row of $W$, while 
	$\alg^{m,n}_{m+1,i}$ is isomorphic to $\RR^n$  and it corresponds to the $i$-th column of $U^*$. 
	Finally, for $i,j < m-1$, $\alg^{m,n}_{i,j}$ 
	corresponds to the $(i,j)$-entry of $V$.
	The involution is the usual matrix transpose. We have 
	$e_{m+1} = 	\begin{psmallmatrix}
	0 & 0 \\ 0 & I_n
	\end{psmallmatrix}$ and, for $i < m+1$, $e_i$ is the diagonal matrix having $1$ in the $(i,i)$-entry and zero elsewhere. 
	The T-algebra trace function $\tr(x)$
	satisfies $\tr(x) = \tr(V)+\alpha$, where $\tr(V)$ is the usual matrix trace of $V$.

	With that, axioms \ref{ax:1} and 
	\ref{ax:2} of T-algebra definition are satisfied.
	Checking axioms \ref{ax:3}-\ref{ax:7} is cumbersome but  relatively straightforward given that the multiplication only differs from the usual matrix multiplication in computations involving the block $\alg^{m,n}_{m+1,m+1}$. 
	For this case, it is more convenient to check \ref{ax:6} and \ref{ax:7} by verifying, equivalently, that 
	for upper triangular $t,u,w \in \UT(\alg^{m,n})$ we have $t(uw) = (tu)w$ and $t(uu^\T) = (tu)u^*$ respectively, see \cite[Remarks~6 and 7]{CH09}. 
	In particular, \ref{ax:6} is immediate given that the multiplication in 	$\alg^{m,n}$ coincides with the usual matrix multiplication for upper triangular matrices.	
	
	%
	
	
	This formulation allow us to apply Theorem~\ref{th:hface} to $\mathcal{M}_{m,n}$. 
	Every principal subalgebra $\alg^{m,k}_I$ is either a usual matrix algebra, if $m+1 \in I$, or a smaller $\alg^{k,n}$ with $k\leq m$ if $m+1 \not \in I$. 
	This implies that all faces of  $\mathcal{M}_{m,n}$ are either linearly isomorphic to $\psdcone{k}$  or to $\mathcal{M}_{k,n}$  for $k \leq m$. Furthermore, $\mathcal{M}_{m,n}$ is projectionally exposed.
	
\end{example}

\paragraph{Some consequences and discussion}
The results so far together with Algorithm~\ref{alg:cholesky} provide an algorithmic way of identifying extreme rays of $\closure \stdHC$.
\begin{corollary}[Identifying extreme rays]\label{col:ext}
Let $x \in \closure \stdHC$ and write $x = tt^*$ for some proper $t \in \UT_{+}$. Then $x$ generates an extreme ray if and only if 
$\rho_{i}(t_{ii}) > 0$ for one and exactly one $i \in \{1,\ldots,r\}$.
\end{corollary}
\begin{proof}
Let $\stdFace = \minFace(x, \closure \stdHC)$. We observe that $x$ generates an extreme ray if and only if the dimension of $\stdFace$ is one.
By Theorem~\ref{th:hface}, $\stdFace$ is linearly isomorphic to 
$\closure \stdCone(\alg_{I})$, where $I = \{i \mid \rho_i(t_{ii}) = 0 \}$.
Since $e_i \in \closure \stdCone(\alg_{I})$ for $i \not \in I$, the dimension of $\closure \stdCone(\alg_{I})$ is at least $r - |I|$. 
So the only way that $\stdFace$ can be an extreme ray is if $r - |I| = 1$, which happens if and only if $\rho_{i}(t_{ii}) > 0$ for one and exactly one $i \in \{1,\ldots,r\}$. Conversely, if $r - |I| = 1$ then a direct verification shows that the dimension of $\closure \stdCone(\alg_{I})$ is one, so $\stdFace$ is an extreme ray.
\end{proof}

Given a closed convex cone $\stdCone$, a \emph{chain of faces of length $\ell$} is a sequence of faces of $\stdCone$ such that $\stdFace _1 \subsetneq \cdots \subsetneq \stdFace_{\ell}$. The length of a longest chain of faces of $\stdCone$, denoted by $\ell_{\stdCone}$, is an important quantity in the study of the regularization technique \emph{facial reduction} \cite{BW81,WM13,Pa13_2}. In particular, it can be used to upper bound the number of steps required by facial reduction algorithms.
Through this connection, $\ell_{\stdCone}$ appears in different contexts, e.g., as a way to upper bound the length of infeasibility certificates \cite[Theorem~4]{LiuPataki}.
The quantity $\ell_{\stdCone}$ also appears in the study of the expressive power of cones \cite{Sa20}.
Next, we compute $\ell_\stdCone$ for a homogeneous cone.


\begin{corollary}\label{col:chain}
Let $\closure \stdHC$ be a homogeneous cone of rank $r$, 
then $\ell_{\stdCone} = r+1$.
\end{corollary}
\begin{proof}
Take any chain of faces $\stdFace _1 \subsetneq \cdots \subsetneq \stdFace_{\ell}$ of length $\ell$. By item~$(ii)$ of Theorem~\ref{th:hface}, each $\stdFace_i$ is a homogeneous cone of rank smaller than that of $\stdFace_{i+1}$. As $\closure \stdHC$ has rank $r$, we have $\ell \leq r+1$.
Using principal faces
we can construct a chain of faces of length $r+1$ as follows.
\[
\{0\}=\closure\stdCone(\alg_{\{1,2,\ldots,r\}})   \subsetneq \closure\stdCone(\alg_{\{1,2,\ldots,r-1\}}) \subsetneq \cdots \subsetneq \closure \stdCone(\alg_{\{1\}}) \subsetneq  \closure \stdHC.
\]
\end{proof}

The automorphism group of  $\closure \stdHC$ acts on the set of (non-empty) faces in a natural way. 
It turns out that this action has  finitely many orbits and each orbit contains a unique principal face.
\begin{corollary}[Orbits of faces]\label{col:orbits}
Let $\stdFace(\closure \stdHC)$ denote the set of faces of a rank $r$ homogeneous cone $\closure \stdHC$. The action of the triangular automorphisms on $\stdFace(\closure \stdHC)$ has $2^r$ orbits and each orbit has a unique representative among principal faces.
\end{corollary}
\begin{proof}
By Theorem~\ref{th:hface} each $\stdFace \face \closure \stdHC$ is isomorphic to some principal face under a triangular automorphism. Since there are at most $2^r$ principal faces, this is also an upper bound on the number of orbits. 
It only remains to show that there are indeed $2^r$ orbits which can be done by arguing that each orbit contains at most one principal face.

Suppose that  $Q_t$ is a triangular automorphism and 
$Q_t (\closure \stdCone(\alg_{I})) =    \closure \stdCone(\alg_{J})$ holds.
Then, $|I| =|J|$ must also hold since rank is preserved by automorphisms by Proposition~\ref{prop:rank}.
Let $e_I$ and $e_J$ be the identity elements of $\alg_I$ and $\alg_J$ respectively. We have $e_I \in \reInt \closure \stdCone(\alg_{I}) = \stdCone(\alg_I)$, so $Q_{t}e_I$ is an element of $\stdCone(\alg_J)$ since linear transformations map relative interiors into relative interiors, e.g., see \cite[Theorem~6.6]{rockafellar}.

As $\stdCone(\alg_J)$  is a homogeneous cone, there exists 
a triangular automorphism $Q_{\tilde t}$ of $\stdCone(\alg _J)$ that maps $Q_te_I$ to $e_J$.
Then, Proposition~\ref{prop:sub_aut} implies 
that there exists a triangular automorphism $Q_{\bar t}$ of $\stdHC$ that maps $Q_te_I$ to $e_J$.
Therefore, $u \coloneqq \bar t t$ is such 
that $Q_u$ is a triangular automorphism  of
$\stdHC$ satisfying $Q_u (e_I ) = e_J$, by \eqref{eq:qut}.
Which in view of \eqref{eq:qut2} implies that
\[
Q_u (e_I) = (ue_I)(ue_I)^\T = e_J = e_J e_J^\T.
\]
Following similar computations as in the proof of Lemma~\ref{lem:qei}, we see that $ue_I$ is a proper matrix.
Since the decomposition in proper matrices is unique (Proposition~\ref{prop:chol}), we conclude that $ue_I = e_J$ must hold.
Next, we verify that this implies $I = J$.

Recalling that $e_I = \sum_{i \not \in I} e_i$, we have that
$(ue_{I})_{ij} =  \sum _{k=1}^r u_{ik} (e_{I})_{kj}$ is $0$ if $j  \in I$ and $u_{ij}$ if $j \not \in I$.
Since $ue_I = e_J$ must hold, we have $u_{ij} = 0$ for $i\neq j$, i.e., 
$u$ only has nonzero components on its ``diagonal''.
Finally, since $u \in \UT_{++}$, $ue_I = e_J$ forces that $I = J$.
\end{proof}
Corollary~\ref{col:orbits} is about the orbits under triangular automorphisms. 
If we consider the orbits under the \emph{full} automorphism group, the number of orbits never increases, which implies that, in particular, it is still finite and principal faces can still be taken to be representatives of orbits. However, it may happen that each orbit may contain several principal faces\footnote{For example, all extreme rays of $\psdcone{n}$ (which correspond to the faces generated by rank-$1$ PSD matrices) are in the same orbit under the action of the full automorphism group.}.

A finite number of orbits of faces is an interesting property that is not shared by all finite dimensional convex cones. For example, that is not the case for the $p$-cones for $p \in (1,\infty)$, $p \neq 2$ in dimension at least $3$, as is implied by \cite[Corollary~12]{IL19}\footnote{
	Corollary~12 therein implies that each orbit of the action of the automorphism group on the set of faces is finite. 
	As there infinitely many faces, there  must be an infinite number of orbits.}.

Next, we take a look at projectional exposedness. It was shown in \cite[Proposition~33]{L21} that all symmetric cones are orthogonally projectionally exposed. 
In view of Theorem~\ref{th:hface}, the class of cones known to be projectionally exposed includes the homogeneous cones but it seems that the price to pay for losing self-duality is that we can no longer ensure that the projections are  self-adjoint under the inner product that comes from the T-algebra structure.
We will revisit this issue in Proposition~\ref{prop:orth_proj}.

More generally, one may wonder how large the class of projectionally exposed cones is.
Notably, for hyperbolicity cones and spectrahedral cones, two classes of cones larger than homogeneous cones, it is not currently known whether they are projectionally exposed in general. The strongest result so far is that they are amenable, see \cite[Corollary~3.5]{LRS20} and \cite{LRS23}.

\subsection{The dual cone}\label{sec:dual}
A T-algebra for the dual cone $\stdHC^*$ can be obtained as in \eqref{eq:algD}.
Under the bigradation induced by $\algD$, upper triangular elements in $\algD$ correspond to lower triangular elements in $\alg$. 
All the development done so far can thus be translated to the dual side by appealing to \eqref{eq:algD}.
Still, we believe it is useful to explicitly state certain dual results and translate the terminology, since there are certain subtleties.

For convenience, let $\LT_{++}(\alg) \coloneqq \UT_{++}(\alg)^\T$,
$\LT_{+}(\alg) \coloneqq \UT_{+}(\alg)^\T$ and
$\LT(\alg) \coloneqq \UT(\alg)^\T$.
A \emph{triangular automorphism} of $\stdHC^*$ is a linear map of the form $Q_{l}$ for 
$l \in \LT_{++}(\alg)$. Analogous properties to \eqref{eq:qut}, \eqref{eq:qinv}, \eqref{eq:qut2} hold for $l,t \in \LT_{+}$.
For example, $Q_{l}(tt^\T) = (lt)(lt)^\T$ holds for $l,t \in \LT_{+}(\alg)$, see also \cite[Proposition~7]{CH09}.

A potential pitfall is the definition of proper (lower) triangular matrices. 
Applying the definition of proper (upper) triangular matrices to $\algD$ and translating the indices back to $\alg$, we see the ``correct way'' of defining properness. 
Indeed, $t \in \LT_{+}(\alg)$ is said to be \emph{proper} if $\rho_{i}(t_{ii}) = 0$ implies that 
the $i$-th ``column'' (not the row!) of $t$ vanishes\footnote{This is somewhat confusing, so here is a quick example. For $r = 3$, let $t$ be an upper triangular element in $\algD$ and let us look at the third column, which corresponds to $t_{13}, t_{23}, t_{33}$.
In view of \eqref{eq:algD}, if we decompose $t$ in the bigradation of $\alg$ (i.e., express $t$ as sum of elements in the $\alg_{ij}$'s), $t$ corresponds to a lower triangular $\hat t \in \LT(\alg)$ and the third column of $t$ gets mapped to the first column of $\hat t$, so that $t_{33} = \hat t_{11}$, $t_{23} = \hat t_{21}$ and $t_{13} = \hat t_{31} $.   }.
Analogous to Algorithm~\ref{alg:cholesky}, we also have a Cholesky factorization in proper lower triangular matrices. 
Essentially, we apply Algorithm~\ref{alg:cholesky} to $\algD$ and 
translate back the indices to $\alg$ in 
two steps. First, the indices 
in $t_{ij}$ and $y_{ij}$ are replaced with $r+1-i$ and $r+1-j$, respectively and, then we simplify the resulting \emph{for loop}. 
The result is Algorithm~\ref{alg:cholesky_dual}.

%
\begin{algorithm}[h]
	\caption{Dual Generalized Cholesky decomposition in a homogeneous cone }
	\label{alg:cholesky_dual}
	\begin{algorithmic}[1]
		\REQUIRE{$x \in \stdHC^*$}
		\STATE $l \leftarrow 0$, $t \leftarrow 0$, $y \leftarrow x $
		
		\FOR{$i = 1,\ldots,r$}
		\STATE $y \leftarrow y - ll^\T$
		\IF {$\rho_{i}(y_{ii}) > 0$}
		\STATE $t_{ji} \leftarrow y_{ji}/\sqrt{\rho_{i}(y_{ii})}, \quad \forall j \in \{i,\ldots, r\}.$
		\ELSE
		\STATE $t_{ji} \leftarrow 0, \quad \forall j \in \{i,\ldots, r\}.$
		\ENDIF
		\STATE $l \leftarrow$ $i$-th column of $t$
		\ENDFOR
		\RETURN{$t$}
	\end{algorithmic}
\end{algorithm}

The only remaining technical detail is that principal faces need to be replaced with their duals taken \emph{inside the corresponding subalgebra}. 
More precisely, for a principal subalgebra $\alg_I$,
let 
\[
\stdCone^D(\alg_I) \coloneqq \{tt^\T \mid t \in \LT_{++}(\alg_{I})\},
\]
so that 
\begin{equation}\label{eq:dual_principal}
\closure (\stdCone^D(\alg_I)) = \{tt^\T \mid t \in \LT_{+}(\alg_{I})\}
\end{equation}
holds.
A face of $\stdHC^*$ as in \eqref{eq:dual_principal} is said to be a \emph{principal dual face}.
We note that $\closure (\stdCone^D(\alg_I))$ is 
the dual of $\closure \stdCone(\alg_I)$ \emph{inside  $\H(\alg_{I})$}, i.e., 
we have 
$\closure (\stdCone^D(\alg_I)) = \closure \stdCone(\alg_I)^* \cap \H{(\alg_I)}$.
\begin{theorem}\label{th:hface_dual}
	Let $\stdHC$ be a homogeneous cone of rank $r$ 
	and  $\stdFace$ be a proper face of $\stdHC^*$.
	\begin{enumerate}[$(i)$]
		\item $\stdFace$ is projectionally exposed.
		\item $\stdFace$ is a homogeneous cone of rank $s < r$ and		
		there is a triangular automorphism $Q_l$ of 
		$\stdHC^*$ such that $Q_l (\stdFace) = \closure (\stdCone^D(\alg_{I}))$ holds for some nonempty $I \subseteq \{1,\ldots, r\}$.
		Furthermore, letting $x \in \reInt \stdFace$ and writing $x = tt^*$ for a proper $t \in \LT_{+}(\alg)$, we have $I = \{i \mid \rho_i(t_{ii}) = 0 \}$ and $s = r - |I|$.
		\item There is a subgroup $\mathcal{G}$ of automorphisms of $\stdFace$ acting simply transitively on the relative interior of $\stdFace$ such that each $g \in \mathcal{G}$ is obtained through the restriction of a triangular automorphism of $\stdHC^*$ to $\stdFace$.
	\end{enumerate}	
\end{theorem}
\begin{proof}
Following the discussion so far, we apply Theorem~\ref{th:hface} to $\algD$ as in \eqref{eq:algD} and translate the indices back to the bigradation of $\alg$.	
\end{proof}
Analogous to \eqref{eq:utei}, the automorphism that maps $\stdFace$ to a principal dual face can be
obtained by finding $l \in \LT_{++}(\alg)$ satisfying $lt = e_{I}$.

We conclude this subsection by observing that 
conjugate faces in a homogeneous cone have complementary ranks.
\begin{proposition}\label{prop:conj_rank}
Let $\stdHC$ be a homogeneous cone of rank $r$
and  let $\stdFace \face \closure \stdHC$ be a face of rank $s$. 
Then,
$\stdFace^\Delta = \stdHC^* \cap \stdFace^\perp$	has rank $r-s$.
\end{proposition}
\begin{proof}
Let $x \in \reInt \stdFace$.
By Theorem~\ref{th:hface}, there exists a triangular automorphism $Q_{\hat u}$  of $\stdHC$ that maps $\stdFace$ into a principal face 
$\closure \stdCone(\alg_I)$. 
We may assume that $Q_{\hat u}$ maps $x$ to 
$e_I$, see the discussion around \eqref{eq:utei}.
Letting $u \coloneqq \hat u^{-1}$, $Q_{u} = Q_{\hat u}^{-1}$ does the opposite: it maps 
$e_I$ to $x$.
Recalling that $\stdHC^* \cap \stdFace^\perp = 
\stdHC^* \cap \{x\}^\perp$,
we have by \eqref{eq:qadj}
\[
\stdHC^* \cap \{x\}^\perp = \{y \in \stdHC^* \mid \inProd{Q_{u}(e_{I})}{y} = 0\} = 
\{y \in \stdHC^* \mid \inProd{e_{I}}{Q_{u^\T}(y)} = 0\}.
\]
We have that $e_I = \sum _{i \not \in I} e_i$ and $|I| = r-s$.
$Q_{u^\T}$ is a triangular automorphism 
of $ \stdHC^*$, so $Q_{u^\T}(y) \in \stdHC^*$ if and only if $y \in \stdHC^*$.
In addition, as the $e_i$'s are in $\closure\stdHC$,
for $ y \in \stdHC^*$ we have
\[
\inProd{e_{I}}{Q_{u^\T}(y)} = 0 \Leftrightarrow
\inProd{e_{i}}{Q_{u^\T}(y)} = 0, \forall i \not \in I.
\]

Therefore, from Proposition~\ref{prop:row_vanish}, for $y \in \stdHC^*$, we
have  $\inProd{e_{I}}{Q_{u^\T}(y)} = 0$ if and only if every row and column of $Q_{u^\T}(y)$ that is
indexed by $i$  not in $I$ vanishes.
This happens if and only if  $Q_{u^\T}(y)$ is in the principal \emph{dual} face 
$\closure (\stdCone^D(\alg_J))$, where
$J \coloneqq \{1,\ldots, r\} \setminus I$.
In particular, 
\[
\stdHC^* \cap \{x\}^\perp = Q_{u^\T}^{-1}(\closure (\stdCone^D(\alg_J))).
\]

That is, $\stdHC^* \cap \{x\}^\perp$ is 
linearly isomorphic to a principal \emph{dual} face of rank $r - |J| = |I| = r-s$, which concludes the proof.
\end{proof}
We note that the proof of Proposition~\ref{prop:conj_rank} gives a recipe for determining the conjugate face $\stdFace^\Delta$. 
If $x \in \reInt \stdFace$ and $x=tt^\T$ for a proper $t$, 
we let  $u \in \UT_{++}$ be such that $ue_{I} = t$, where $I \coloneqq \{i \mid \rho_{i}(t_{ii}) = 0\}$. 
With that, $Q_{u^*}$ is an automorphism of $\stdHC^*$ that maps 
$\stdFace^\Delta$ to $\closure (\stdCone^D(\alg_J))$, where 
$J \coloneqq  \{i,\ldots, r\} \setminus I$.

\subsection{Other approaches and related results}\label{sec:other}
The T-algebra framework is not the only the tool suitable for the analysis of homogeneous cones. In this subsection, we discuss some other possibilities, related results and justify our usage of T-algebras.

\paragraph{Siegel cones}
Let $t \in \Re, u \in \Re^{n-1}, x \in \mathcal{S}^{n-1}$.
The matrix $\begin{psmallmatrix}
t & u \\ u^\T & x
\end{psmallmatrix}$ is positive definite if and only if $t > 0$ and its Schur complement $x - u^\T u/t$ is positive definite. Or, equivalently, $t > 0$ and $xt -u ^\T u $ is positive definite.
Notably, ``$xt - u^\T u$ being positive definite'' is a constraint over $(n-1)\times (n-1)$ matrices.
This allow us to construct $\psdcone{n}$ recursively via smaller dimensional PSD cones and we can repeat this argument until we reach $\psdcone{1}$, which is isomorphic to the nonnegative orthant. 
We can also view this construction in a bottom-up fashion and, starting from the nonnegative orthant, the cone $\psdcone{n}$ arises by $n$ applications of Schur complements. 

A similar technique is available for homogeneous cones.
Given a regular homogeneous cone $\stdCone \subset \RR^n$ and a bilinear form $B:\RR^p\times\RR^p\to \RR^n$, the \emph{Siegel cone of $\stdCone$ and $B$} is defined as
$$\SC(\stdCone,B)\coloneqq \closure \{(x,u,t)\in\RR^{n}\times\RR^p\times\RR \mid t> 0,~tx-B(u,u)\in \stdCone\}$$
and it is homogeneous if $B$ is a so-called ``homogeneous $\stdCone$-bilinear form'', whose precise definition we omit.
Importantly, every closed regular homogeneous cone $\stdCone$ of dimension at least two is linearly isomorphic to a Siegel cone $\SC(\stdCone_1,B_1)$, where $\stdCone_1$ is a lower dimensional homogeneous cone $\stdCone_1$ and $B_1$ is a homogeneous $\stdCone_1$-bilinear form (see \cite{Ro66}, \cite[Chapter~2]{Gin92}, \cite[Section~3]{GT98}). 
In this way, starting from the nonnegative orthant, every closed regular homogeneous cone can be obtained by repeated applications of this construction.

T-algebras and Siegel cones are complementary approaches that generalize distinct aspects of PSD matrices.
While T-algebras provide a generalized notion of Cholesky factorization, the Siegel cone construction generalizes Schur complements. 

Naturally, arguments involving Siegel cones tend to be inductive. If we wish to show that all homogeneous cones have some property up to linear isomorphism, all we need to show is that the nonnegative orthant satisfies it (the easy part) and that if a homogeneous cone $\stdCone$ satisfies it then $\SC(\stdCone,B)$ satisfies it too (the delicate part).
For instance, this is how the facial exposedness of homogeneous cones was established in \cite{TT03}. 
Similarly, although we did not explore this possibility, 
it is likely possible to prove the projectional exposedness of homogeneous cones via an inductive argument using Siegel cones.

The Siegel cone approach allows very elegant proofs, but in our opinion, it is less convenient for describing a given arbitrary homogeneous cone. 
After all, while we may use Schur complements to prove properties of PSD matrices, we rarely think of a $n\times n$ positive definite matrix as $n$ Schur complements taken in sequence.

Similarly, for an arbitrary homogeneous cone $\stdCone$, it is useful to furnish the ambient space with an appropriate basis and reason about membership in $\stdCone$ in terms of this basis in a direct (i.e., non-recursive) way, which is what the T-algebra framework gives us.
  Of course, it is not impossible to do something analogous under the Siegel cone framework. 
  Doing so, however, naturally leads to algebras of generalized matrices that play a very similar role to the triangular matrices in the T-algebra construction, see, for example, \cite[pg.75-77 and \S 1.Theorem~2]{Gin92}.
  In our opinion, going along this route is not significantly different from using T-algebras from the beginning and has a similar amount of algebraic complexity.
  
We emphasize that rather than furnishing the shortest possible proofs of our statements, we are interested in convenient descriptions of the objects appearing in our discussion. 




\paragraph{Matrix realizations}
Let $\stdCone$ be a homogeneous cone. 
Then $\stdCone$ is \emph{spectrahedral}, which means that it is linearly isomorphic to an intersection of the form $\psdcone{n} \cap \mathcal{L} $, where $\mathcal{L}$ is a subspace of the space of $n\times n$ real symmetric matrices $\mathcal{S}^n$. Here, we refer to $\psdcone{n} \cap \mathcal{L} $ as a \emph{matrix realization} of $\stdCone$.
As discussed in \cite[Section~6]{TV23}, the existence of matrix realizations of homogeneous cones is implicit in earlier works by Vinberg \cite{V63} and Rothaus \cite{Ro66} which provided the theoretical backing for later approaches \cite{CH03,FB02,Ishi06,Ishi15}.

The geometric properties of an arbitrary spectrahedral cone can be quite different from homogeneous cones. For example, as mentioned previously, it is not known whether arbitrary spectrahedral cones are projectionally exposed.
Therefore, as far as we can see, results such as Theorem~\ref{th:hface} do not seem to follow simply from the {mere existence} of matrix realizations. 

Nevertheless, discussions on matrix realizations of homogeneous cones often include results on how to obtain a group of linear transformations that act transitively into the relative interior of $\psdcone{n} \cap \mathcal{L}$, e.g., see \cite[Theorem~3]{Ishi15}. 
With that, although we have not worked out the details, it seems entirely plausible to give alternative proofs of the results we have shown so far by making use of matrix realizations and appropriate descriptions of transitive subgroups of the automorphism group.
One advantage of this approach is that it does not use T-algebras, thus bypassing some of the algebraic unpleasantness and bringing the discussion into the more familiar territory of PSD matrices.

A  disadvantage is that matrix realizations may be wasteful because for an arbitrary homogeneous cone $\stdCone$ of dimension $n$, in the worst case, we may need  $n \times n$ PSD matrices to represent it, e.g., see \cite[Corollary~4.3]{CH03} and \cite[Theorem~4]{Ishi11}.
This worst case is indeed achieved, e.g., for $\stdCone = \Re^n_+$ as we may verify that a necessary condition for $\psdcone{m} \cap \mathcal{L}$ to be linearly isomorphic to $\Re^n_+$ is that $m \geq n$ holds\footnote{The longest chain of faces of $\psdcone{m}$ is $m+1$ and the faces of $\psdcone{m} \cap \mathcal{L}$ are of the form $\stdFace \cap \mathcal{L}$ where $\stdFace$ is a face of $\psdcone{m}$ satisfying $(\reInt \stdFace) \cap \mathcal{L} \neq \emptyset$. Furthermore, suppose that $\stdFace_1$ and $\stdFace_2$ are two such faces of $\psdcone{m}$ and that the proper inclusion $\stdFace _1 \cap \mathcal{L} \subsetneq \stdFace _2 \cap \mathcal{L}$ holds.
	Then, we have $\reInt(\stdFace _1 \cap \mathcal{L}) = \reInt(\stdFace _1) \cap \mathcal{L} \subseteq \stdFace _2 \cap \mathcal{L}$, where the first equality follows from \cite[Corollary~6.5.1]{rockafellar}. That is, $\reInt(\stdFace _1)$ passes through $\stdFace_2$, which implies that $\stdFace_1 \subsetneq \stdFace_2$.
	As the longest chain of faces of $\Re^n_+$ is $n+1$, the only way we can fit a chain of faces of length $n+1$ into $\psdcone{m} \cap \mathcal{L}$  is if $m \geq n$ holds.
}.

%

When considering different formulations for the same optimization problem, if all else is the same, it is preferable to use the one for which the ambient space has dimension as small as possible. 
That is one of the reasons why it is rare to solve a second-order cone program by resorting to semidefinite programming formulations, see also related comments in \cite[Section~1]{CH03}.

Another undesirable fact is that a cone and its dual may need different $n$'s in their matrix realizations.  For example the Vinberg cone and its dual are both  homogeneous cones of rank~$3$. While the latter has a natural matrix representation over $3 \times 3$ matrices, the former cannot be written in the form $\psdcone{3}\cap \mathcal{L}$, see this footnote\footnote{We only provide a sketch of the argument. 
	The Vinberg cone is defined as $
	 \stdCone \coloneqq \left\{(y_1,y_2,y_3,y_4,y_5) \in \mathbb{R}^5 \mid \begin{psmallmatrix}
	y_1 & y_2 \\
	y_2 & y_4 
	\end{psmallmatrix} \in \psdcone{2}, \begin{psmallmatrix}
	y_1 & y_3 \\
	y_3 & y_5 
	\end{psmallmatrix} \in \psdcone{2}  \right\}$. 
	Nonzero polynomials of minimal degree that vanish on the boundary of $\stdCone$ must have degree $4$, see \cite[Section~4.1]{GIL24}. 
	If $\stdCone$ were linearly isomorphic to $\psdcone{3}\cap \mathcal{L}$ for some subspace $\mathcal{L}$, 
	the restriction of the determinant function to $\psdcone{3}\cap \mathcal{L}$
	 would lead to a polynomial of degree at most $3$ that vanishes on the boundary of $\stdCone$, which is a contradiction. On the other hand, $\stdCone^*$ is $\left\{(x_1,x_2,x_3,x_4,x_5) \in \mathbb{R}^5 \mid \begin{psmallmatrix}
	 	x_1 & x_2 & x_3\\
	 	x_2 & x_4 & 0\\
	 	x_3 & 0   & x_5\\
	 \end{psmallmatrix} \in \psdcone{3} \right\}$, which indeed has a matrix realization over $3 \times 3$ matrices. Finally, we remark that whether $\stdCone$ or $\stdCone^*$ is referred as the ``Vinberg cone'' depends on the reference. Ishi calls $\stdCone$ the ``Vinberg cone'' \cite[Section~5.2]{Ishi00}, while, up to a permutation of rows and columns, Tun\c{c}el and Vanderberghe call $\stdCone^*$ the ``Vinberg cone'' \cite[Example~6.3]{TV23}.}.

The T-algebra representation, albeit cumbersome, is economical since both
$\closure \stdHC $ and its dual $(\closure \stdHC)^*$ are realized as cones of generalized matrices in the same space $\HE(\alg)$, which has the same dimension of $\closure \stdHC$ and allows, for example, an intuitive discussion of the complementarity between ranks of conjugate faces as in Proposition~\ref{prop:conj_rank}.

\paragraph{Connection to Ishi's results}
Here we will describe the connections of our discussion so far with a
general result by Ishi \cite{Ishi00}.
Let $\stdCone$ be an open homogeneous cone and $H$ a group that acts linearly and simply transitively on $\stdCone$, see \cite{Ishi00} for more details on the assumptions. 
Then, Ishi shows that there are $2^r$ distinct orbits of the action of $H$ onto $\closure \stdCone$, see \cite[Theorems A and 3.5]{Ishi00}. 
This is done by invoking the theory of normal j-algebras, some Lie algebra theory and showing that the action of $H$ onto $\closure \stdCone$ can be expressed via generalized triangular matrices.
Although the algebraic framework is different, the development in \cite{Ishi00} is somewhat close in spirit to what we have done here.
In particular, translating to our language, Ishi proved that 
\[
\closure \stdHC = \bigcup _{I \subseteq \{1,\ldots, r\}}\mathcal{O}_{I},
\]
where $\mathcal{O}_I$ is the orbit of $e_{I} = \sum _{i \not\in I} e_i$ under the action of the triangular automorphisms and the union is disjoint.
This is essentially equivalent to Corollary~\ref{col:orbits}.
Although there is no discussion of facial structure in \cite{Ishi00},
using Ishi's result as a starting point, it seems feasible to give another proof of  Theorem~\ref{th:hface}.
But in this case, some  arguments regarding projectional exposedness and the realization of faces as homogeneous cones would still be necessary.
While the discussion here has some overlap with \cite{Ishi00}, 
our point of view is different and we emphasize
certain concrete aspects such as Algorithm~\ref{alg:cholesky}, how to explicitly identify faces, exposing vectors and the automorphisms needed to reveal them as in \eqref{eq:utei} and \eqref{eq:exp}.
Our development is also entirely done in the T-algebraic language, which is important from an algorithmic point of view, as it gives easy access to a self-concordant barrier of optimal parameter, e.g., see \cite[Section~3.1]{CH09}, \cite[Theorem~4.1]{GT98}.


%
\section{Applications to homogeneous chordality}\label{sec:chord}

For this section, in contrast to Remark~\ref{rem:usual} and the previous sections, given a real matrix $a \in M^{n\times n}$, we  return to the usual matrix convention where $a_{ij}$ indicates the $(i,j)$-entry of $a$.
Then, given a graph $G$ with vertices $V=\{1,...,n\}$ and edges $E$, we associate to it a convex cone $\psdG{G} \subseteq \psdcone{n}$ given by
\begin{equation*}
\psdG{G}\coloneqq\{x \in  \psdcone{n} \, | \, x_{ij}=0 \textrm{ for all } i\not=j \textrm{ such that } \{i,j\} \not \in E\}
\end{equation*}
and we denote by $\mathcal{S}(G)$ the space spanned by $\psdG{G}$.
If $G$ is chordal, the cone $\psdG{G}$ is
generated by its rank one matrices, e.g., \cite[Theorem~2.3]{AHM88}. 
Chordal graphs have found extensive applications in solving semidefinite programs efficiently, e.g., \cite{FKMN01,NFFKM03,VA15}.
In particular,  it is well-known that, by reordering rows and columns if necessary, positive definite matrices with a chordal sparsity pattern have Cholesky factorizations that respect the underlying sparsity pattern, i.e., factorizations with ``zero fill-in''.

When $G$ not only is chordal but also does not contain induced subgraphs isomorphic to a  path on $4$ vertices, by  \cite[Theorem~A]{Ishi13} they are also homogeneous and we will refer to such a graph $G$ as being a \emph{homogeneous chordal graph}. 
This class of cones was also extensively studied in \cite{TV23}. 
A graph that does not contain  an induced subgraph isomorphic to a  path on $4$ vertices is also called a \emph{cograph}, so, alternatively, a 
homogeneous chordal graph can be defined as a chordal cograph.
Cholesky factorizations under homogeneous chordal sparsity have stronger properties when compared to plain chordal sparsity, e.g., inverses also respect the underlying sparsity pattern \cite[item~(ii) of Theorem~3.1]{TV23}.

%

The T-algebra underlying $S_+(G)$ is relatively simple, but the ordering of the vertices matter. 
Chordal homogeneous graphs have a very interesting ordering, called a \emph{trivially perfect elimination ordering}. 
This means that the following two properties are satisfied.
\begin{enumerate}[$({o}1)$]
	\item \label{el1} If $\{i,j\}, \{i,k\} \in E$ and 
	$i < j < k$ hold, then $\{j,k\} \in E$.
	\item \label{el2} If $\{i,j\}, \{j,k\} \in E$ and 
	$i < j < k$ hold, then $\{i,k\} \in E$.
\end{enumerate}
See more details in \cite[Section~2.3]{TV23}.
An ordering that only satisfies \ref{el1} is called a \emph{perfect elimination ordering} and all chordal graphs possess one. 
An ordering satisfying both \ref{el1} and \ref{el2} is a privilege of homogeneous chordality.
In what follows, we will assume that the vertices are ordered following 
a {trivially perfect elimination ordering}.

Let $\alg(G) \coloneqq \{a \in M^{n\times n} \mid a_{ij}= a_{ji} = 0 \textrm{ for all } i\not=j \textrm{ such that } \{i,j\} \not \in E \}$. 
We take the straightforward decomposition $\alg(G) = \bigoplus _{i,j=1}^r \alg _{ij}$ where for $\{i,j\} \in E$ or $i = j$, $\alg _{ij} \cong \RR$ is the subspace of $M^{n\times n}$ of matrices that are zero outside the $(i,j)$-entry, i.e., $a \in \alg_{ij}$ if and only if $a_{kl} = 0$ for $(k,l) \neq (i,j)$.
Otherwise, we set $\alg _{ij}$ to the zero subspace $\{0\}$.
Overall, $\alg(G)$ is the set of $n\times n$ matrices with the sparsity pattern prescribed by $G$. The involution is defined as the usual matrix transposition and the multiplication is defined as matrix multiplication, followed by the projection onto the sparsity pattern defined by $G$.
We denote such a projection by $\pi_G$ so that 
$\pi_G(a)$ is the result of ``zeroing out'' the entries of $a_{ij}$ for which $\{i,j\} \not \in E$ and $i\neq j$. With that, denoting the multiplication in $\alg(G)$ by $\star$, we have for $a, b \in \alg(G)$
\[
a\star b = \pi_G(ab)
\]
and the $\alg_{ij}$ satisfy \eqref{eq:m_rules} so it is indeed a bigradation of $\alg(G)$. We also recall that $\pi_G$ is self-adjoint with respect to the trace inner product.


To see that this is indeed a T-algebra we have to check the seven axioms.  For the sake of clarity and preciseness, we recall once more that for $a \in \alg(G)$, $a_{ij}$ denotes the $(i,j)$-entry of $a$ rather than the element of $\alg_{ij}$ corresponding to $a$. This is a subtle, but luckily, unimportant distinction\footnote{After all, the $(i,j)$-entry of $a \star b$ corresponds exactly to the $\alg_{ij}$ component of $a \star b$ in the bigradation of $\alg(G)$.}.

Axioms \ref{ax:1}, \ref{ax:2}, \ref{ax:3} and \ref{ax:5} follow directly from the basic definitions and properties of the usual matrix multiplication, since the projection onto the sparsity pattern plays no role in them.

Axiom \ref{ax:4} needs a little more thought. We have to check that for matrices $a,b,c \in \alg(G)$ we have 
$\tr((a\star b)\star c) = \tr(a\star(b \star c))$. As $\pi_G$ never zeroes diagonal entries, this is equivalent to checking that $\tr(\pi_G(ab)c)=\tr(a \pi_G(bc))$ holds. 
First we prove $\tr(\pi_G(ab)c)=\tr(abc)$ by showing that $\pi_G(ab)c$ and $abc$ have the same diagonal entries. 
Note that if $[\pi_G(ab)]_{ij}$ is different from $[ab]_{ij}$, this means $\{i,j\}$ is not in $E$, so $c_{ji}=0$. Therefore 
$$[\pi_G(ab)c]_{ii}=\sum_{j=1}^n [\pi_G(ab)]_{ij} c_{ji}=\sum_{j=1}^n [ab]_{ij} c_{ji} =[abc]_{ii}.$$
The same works for $a \pi_G(bc)$ so we have axiom \ref{ax:4}.

For axioms \ref{ax:6} and \ref{ax:7} the ordering of the vertices is quite essential.  
Translating to the usual linear algebra language, \ref{ax:6} amounts to requiring that 
for $a \in \mathcal{A}_{ij}, b \in \mathcal{A}_{jk}, c \in \mathcal{A}_{kl}$ such that $1 \leq i \leq j \leq k \leq l \leq r$ we have
\begin{equation}\label{eq:ax6aux}
a \star (b \star c) = (a \star b) \star c.
\end{equation}
We start by observing that if, say, $i = j$, since $b \star c \in \mathcal{A}_{jl}$ and $b \in \mathcal{A}_{jk}$, we would have 
$a \star (b \star c) = a_{ii}(b \star c) = (a_{ii} b) \star c = (a \star b) \star c$, by \ref{ax:2}. By similar arguments, \eqref{eq:ax6aux} is true if $j = k$ or $k = l$.
Therefore, in order to check  \ref{ax:6}, we may then assume that $i < j < k < l$.
In this case,  if $\{i,j\}, \{j,k\}$ or $\{k,l\}$ is not in $E$, then $a$, $b$ or $c$ would be zero, so both sides of \eqref{eq:ax6aux} would be  zero. But if they are all in $E$ then by \ref{el2} all the possible edges between $\{i,j,k,l\}$ are in $E$, so the projection $\pi_G$ will not have any effect and we simply have the associativity of matrix multiplication.

For axiom \ref{ax:7} we proceed similarly. Let 
 $a \in \mathcal{A}_{ij}, b \in \mathcal{A}_{jk}, c \in \mathcal{A}_{lk}$ such that $1\leq i \leq j \leq k \leq r$, $1 \leq l \leq k \leq r$ and we will check 
 that
\begin{equation}\label{eq:ax7aux}
a \star (b \star c^*) = (a \star b) \star c^*,
\end{equation}
which implies \ref{ax:7}.
We start by noting that, analogously, 
if $i = j$, $j = k$ or $l = k$ holds, then the validity of \eqref{eq:ax7aux} follows from \ref{ax:2}.
So we consider the case where  $i < j < k $ and 
$l < k$. If $\{i,j\}, \{j,k\}$ or $\{l,k\}$ is not in $E$ then both sides of \eqref{eq:ax7aux} are zero, so 
 we may assume that these edges are all in $E$.
 Then, \ref{el2} implies that $\{i,k\} \in E$. 
We have five possibilities.
If $j = l$, we already have all edges between 
$\{i,l,k\} = \{i,j,k\}$ in $E$, so \eqref{eq:ax7aux} is true because $\pi_G$ does not zero any entries. 
The case $i = l$ is analogous as we have all edges between $\{i,j,k\} = \{l,j,k\}$. 
The other three possibilities are $i < j  < l < k $,
$i < l < j < k$ or $l < i < j < k$.
In either case, if  $\{i,l\} \not \in E $, then both sides of \eqref{eq:ax7aux} are zero, so 
we may assume that $\{i,l\} $ is in $E$.
Then, since $\{i,j\},\{i,l\} \in E$, we have $\{j,l\} \in E$ by either \ref{el1} or \ref{el2}
and we conclude that all possible six edges between $\{i,j,k,l\}$ are in $E$.
Therefore, the projection $\pi_G$ does not zero any elements in all the products appearing in \eqref{eq:ax7aux} and we have the usual associativity of matrix multiplication, which implies \eqref{eq:ax7aux}.

%
%
%
%
\medskip

The inner product that comes with this T-algebra structure (see \eqref{eq:innprod}) coincides with the usual matrix trace inner product since
$\tr(a\star b^\T) = \tr(\pi_G(ab^\T) ) = \tr(ab^\T)$.

Next we will verify that $\psdG{G}$ coincides with the dual cone 
$\stdCone(\alg(G))^*$.
For simplicity we write $\UT_{++}(G) \coloneqq \UT_{++}(\alg(G))$,
$\UT_{+}(G) \coloneqq \UT_{+}(\alg(G))$,
$\UT(G) \coloneqq \UT(\alg(G))$,
$\LT_{++}(G) \coloneqq \UT_{++}(G)^\T$,
$\LT_{+}(G) \coloneqq \UT_{+}(G)^\T$ and
$\LT(G) \coloneqq \UT(G)^\T$.
In particular, $\LT_{+}(G)$ and $\UT_{+}(G)$ are, respectively, lower and upper triangular matrices with nonnegative diagonal and sparsity pattern determined by $G$.

The dual cone $\stdCone(\alg(G))^*$ corresponds to elements of the form $l\star l^\T$ for $l \in \LT_{+}(G)$.
An interesting consequence of the  ordering of the graph is that the projection is not needed.
This is a known fact, 
but for the sake of completeness, we verify this. Let $l \in \LT(G)$ and $i > j$, then
\begin{equation}\label{eq:ll_chord}
(ll^\T)_{ij}= 
\sum_{{i \geq k \text{ and } j\geq} k} l_{ik} l_{jk} = \sum_{j\geq k} l_{ik} l_{jk} = l_{ij}l_{jj}+\sum _{k=1}^{j-1} l_{ik} l_{jk}.
\end{equation}
If $\{i,j\} \in E$, then $(\pi_{G}(ll^\T))_{ij} = (ll^\T)_{ij}$, so suppose that $\{i,j\} \not \in E$.
If there exists $k$ such that $k < j < i$ and $l_{ik}$ and $l_{jk}$ are both nonzero, then $\{k,j\}, \{k,i\} \in E$, which implies that $\{i,j\} \in E$ by \ref{el1}, a contradiction.
Therefore,  for every 
$k$ such that $k < j < i$, the term $l_{ik} l_{jk}$ is zero.
This implies that 
$(ll^\T)_{ij}= l_{ij}l_{jj} = 0 = (\pi_{G}(ll^\T))_{ij}$ by \eqref{eq:ll_chord} and the assumption that $\{i,j\} \not \in E$. 
We conclude that, in fact, $l\star l^\T = ll^\T$ holds for  $l \in \LT_{+}(G)$.

This tells us that $\stdCone(\alg(G))^* = \{ll^\T \mid l \in \LT_{+}(G)\} \subseteq \psdG{G}$. 
The converse and a summary of the discussion so far are given in the next theorem. 
We note that items~$(i)$ and $(iii)$ below also follow from the \cite[Theorem~3.1]{TV23}.

\begin{theorem}[T-algebra structure of homogeneous chordal cones]\label{theo:talg_ch}
Let $G = (V,E)$ be a homogeneous chordal cone where the vertices are ordered following a trivially perfect elimination ordering. 
Then $\alg(G)$ is a T-algebra satisfying
\[
\stdCone(\alg(G))^* = \{ll^\T \mid l \in \LT_{+}(G)\} = \psdG{G}
\]
and the following properties
\begin{enumerate}[$(i)$]
	\item $l\star l^\T = ll^\T$, for all $l \in \LT(G)$.
	\item\label{utstar}   $u\star t = ut$, for all $u,t \in \UT(G)$.
	\item \label{ltstar} $l\star t = lt$, for all $l,t \in \LT(G)$.
\end{enumerate}
\end{theorem}
\begin{proof}
It remains to verify the inclusion 
$\psdG{G} \subseteq \{ll^\T \mid l \in \LT_{+}(G)\}$ and items~\ref{utstar} and \ref{ltstar}.
We start with the former. 
The extreme rays of $\psdG{G}$ are generated by rank one matrices, e.g., \cite[Theorem~2.3]{AHM88}.
So, let $v \in M^{n\times 1}$ (i.e., a column vector) be such that $vv^\T$ generates an extreme ray of $\psdG{G}$. Letting 
$J \coloneqq \{i \mid v_{i} \neq 0\}$, we observe  that for every pair $i\neq j$, $\{i,j\} \subseteq J$, we must have $\{i,j\} \in E$, since $(vv^\T)_{ij} \neq 0$.
Let $i$ be the smallest element of $J$.
Replacing $v$ with $-v$ if necessary, we may
assume that $v_i$ is positive.
Then, if we let $l$ be the lower triangular matrix that has $v$ as it is $i$-column and is zero elsewhere, we have $l \in \LT_{+}(G)$, since $l_{ji} = v_{j}$ and $\{j,i\} \in E$ whenever $v_j \neq 0$ and $i \neq j$.
We also have $ll^\T = vv^\T$ which shows that
$\{ll^\T \mid l \in \LT_{+}(G)\}$ contains all the extreme rays of $\psdG{G}$ and leads to the required inclusion.

For item~\ref{utstar}, let $u,t \in \UT(G)$. We have that $ut$ is also  upper triangular, so $(ut)_{ij} = 0$ if $i > j$. 
For  $i < j$ with $\{i,j\} \not \in E$, we have
\[
(ut)_{ij} = u_{ii}t_{ij} + u_{ij}t_{jj}+\sum _{k=i+1}^{j-1} u_{ik}t_{kj} = \sum _{k=i+1}^{j-1} u_{ik}t_{kj},
\]
since $t_{ij} = u_{ij} = 0$. 
The indices of terms in the summation are such that $i < k < j$ holds, so if $u_{ik}$ and $t_{kj}$ are nonzero we would have $\{i,j\} \in E$ by \ref{el2}, a contradiction.
Therefore, $(ut)_{ij} = 0$  for $i\neq j$, $\{i,j\} \not \in E$, i.e., $u\star t = \pi_G(ut) = ut$.
This proves item~\ref{utstar}.
Item~\ref{ltstar} then follows from item~\ref{utstar} by taking adjoints.
\end{proof}
We note that we can also realize $\psdG{G}$ as a ``primal'' homogeneous cone by constructing the  T-Algebra $\algD(G)$ that is dual to $\alg(G)$ as in \eqref{eq:algD}. 
From the ordering induced by $\algD(G)$, matrices in $\psdG{G}$ have upper triangular Cholesky decompositions.



\medskip

Principal subalgebras of $\alg(G)$ are obtained by zeroing out some rows and columns of $\alg(G)$. 
This corresponds to removing vertices of the underlying homogeneous chordal graph, an operation that in fact preserves homogeneous chordality.
For $H = (\hat V, \hat E)$ an induced subgraph of $G$, we define
\[
\stdFace(H) \coloneqq \{x \in \psdG{G} \mid x_{ii} = 0, i \not \in \hat V \}
\]
and, recalling Section~\ref{sec:dual}, we observe that this coincides with the principal 
dual face $\closure (\stdCone^D(\alg(G)_I))$  induced by $I = \{i \mid i \not \in \hat V\}$.
Also, $\stdFace(H)$ (which is a cone of $n\times n$ matrices) is linearly isomorphic to $\psdG{H}$ (which is a cone of $|\hat V| \times |\hat V|$ matrices).
 Theorems \ref{th:hface} and \ref{th:hface_dual} then imply that every face of a homogeneous chordal cone $\psdG{G}$ is isomorphic to a homogeneous chordal cone $\psdG{H}$, where $H$ is an induced subgraph of $G$. 
In more details, we have the following theorem.

\begin{theorem}\label{theo:h_chord_faces}
Let $G = (V,E)$ be a homogeneous chordal graph and let 
$\stdFace \face \psdG{G}$. Then, $\stdFace$ is projectionally exposed and there exists an automorphism 
$Q$ of  $\psdG{G}$ that maps $\stdFace$ to $\stdFace(H)$, where $H$ is an induced subgraph of $G$.
Additionally, if the vertices of $G$ are in a trivially perfect elimination ordering the following items hold.
\begin{enumerate}[$(i)$]
	\item $Q$ can be assumed to be of the form 
	$Q = Q_l$ for some $l \in \LT_{++}(G)$. In particular 
	$Q_l(y)  = lyl^\T$ holds for $y \in \psdG{G}$.
	\item Suppose that $x \in \reInt \stdFace$ and $x = tt^\T$ holds for a proper matrix $t \in \LT_{+}(G)$.
	Then, letting $I = \{i \mid t_{ii} = 0\}$, $H$ is the subgraph obtained by removing the vertices belonging to $I$.
	Furthermore,  $l$ in the previous item can be taken to be any $l \in \UT_{++}(G)$ satisfying $lt = e_I$, where 
	$e_I$ is the diagonal matrix that has $1$ in its $(i,i)$ component if $i \not \in I$ and zero elsewhere.
\end{enumerate}
\end{theorem}
\begin{proof}
Suppose that $G$ follows a trivially perfect elimination ordering.
By Theorem~\ref{theo:talg_ch}, $\alg(G)$ is a T-algebra and Theorem~\ref{th:hface_dual} tells us that $\stdFace$ is projectionally exposed and there is a triangular automorphism $Q_l$ of $\stdCone(\alg(G))^*$ that maps 
$\stdFace$ to a principal dual face of $\psdG{G} = \stdCone(\alg(G))^*$. Principal dual faces are obtaining by zeroing out certain rows and columns, so each principal dual face corresponds to an induced subgraph $H$ of $G$.
Finally,  Theorem~\ref{theo:talg_ch} together with the dual version of \eqref{eq:qut2} imply  that for $y = tt^\T \in \psdG{G}$, where $t \in \LT_+(G)$, we have
$Q_l(t\star t^\T) = Q_{l}(tt^\T) = (lt)(lt)^\T = ltt^\T l^\T = lyl^\T$, which proves item~$(i)$.

Item~$(ii)$ follows from  Theorem~\ref{th:hface_dual} and 
the observation that, as discussed previously, $\stdFace(H)$ coincides with the principal dual face obtained by removing the vertices in $I$.
The fact that $l$ can be taken to be as in the statement
follows from the discussion  around \eqref{eq:utei} applied to the dual algebra and the existence of $l$ follows from  Lemma~\ref{lem:qei}.

Finally, if $G$ is in any other ordering, then $\psdG{G}$ is linearly isomorphic to some $\psdG{G'}$ where $G'$ is just a relabelling of $G$ and is in a trivially perfect elimination ordering. Such a linear isomorphism is obtained by permuting rows and columns, so  $\psdG{G}$ still has the intended properties described in the statement.
\end{proof}
We move on to an example and, after that, we will observe that a part of Theorem~\ref{theo:h_chord_faces} may fail for chordal graphs that are not homogeneous.

\paragraph{An example}
Let $G$ be the following graph.
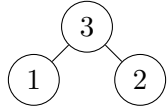
\begin{center}
	\begin{tikzpicture}[main/.style = {draw, circle}] 
	\node[main] (1) {$3$}; 
	\node[main] (2) [below left of =1]{$1$}; 
	\node[main] (3) [below right of =1]{$2$}; 
	\draw (1) -- (2);
	\draw (1) -- (3);
	\end{tikzpicture}
\end{center}
We note that the vertices follow a trivially perfect elimination ordering and that would not be the case if, say, the labels of $3$ and $2$ were exchanged.
With that, we have
\[
\psdG{G} = \left\{\begin{pmatrix}
x_{11} & 0 & x_{13}\\
0 & x_{22} & x_{23}\\
x_{13} & x_{23}   & x_{33}\\
\end{pmatrix} \in \mathcal{S}(G) \mid \begin{pmatrix}
x_{11} & 0 & x_{13}\\
0 & x_{22} & x_{23}\\
x_{13} & x_{23}   & x_{33}\\
\end{pmatrix}  \succeq 0 \right\}.
\]
The nonzero proper principal faces $\psdG{G}$ are the ones obtained from the following  $6$ induced subgraphs.
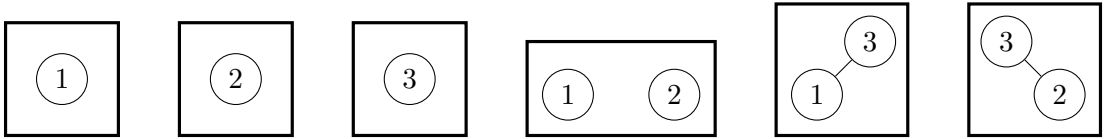
\begin{center}
	\begin{tikzpicture}[main/.style = {draw, circle}] 
	\draw[ very thick] (-0.75,-0.75) rectangle (0.75,0.75);
	\node[main] (1) {$1$}; 
	\end{tikzpicture}\qquad	
	\begin{tikzpicture}[main/.style = {draw, circle}] 
		\draw[ very thick] (-0.75,-0.75) rectangle (0.75,0.75);
	\node[main] (1) {$2$}; 
	\end{tikzpicture}\qquad	
		\begin{tikzpicture}[main/.style = {draw, circle}] 
	\draw[ very thick] (-0.75,-0.75) rectangle (0.75,0.75);
	\node[main] (1) {$3$};
	\end{tikzpicture}\qquad	
	\begin{tikzpicture}[main/.style = {draw, circle}] 
	\draw[ very thick] (-1.25,-1.25) rectangle (1.25,0);
	\node[main] (2) [below left of =1]{$1$}; 
	\node[main] (3) [below right of =1]{$2$}; 
	\end{tikzpicture}\qquad 
	\begin{tikzpicture}[main/.style = {draw, circle}] 
	\draw[ very thick] (-1.25,-1.25) rectangle (0.5,0.5);
	\node[main] (1) {$3$}; 
	\node[main] (2) [below left of =1]{$1$}; 
	\draw (1) -- (2);
	\end{tikzpicture}\qquad
		\begin{tikzpicture}[main/.style = {draw, circle}]
		\draw[ very thick] (-0.5,-1.25) rectangle (1.25,0.5); 
	\node[main] (1) {$3$}; 
	\node[main] (3) [below right of =1]{$2$}; 
	\draw (1) -- (3);
	\end{tikzpicture}
\end{center}
However, since $\psdG{G}$ is not polyhedral, it has infinitely many faces that are not subfaces of the cones obtained by those induced subgraphs. 
For example, let $x \in \psdG{G}$ be such that
\[x=\begin{pmatrix}
1 & 0 & 1\\
0 & 1 & 1\\
1 & 1   & 2\\
\end{pmatrix}
\]
and let us consider the problem of determining $\stdFace = \minFace(x, \psdG{G})$, i.e., the minimal face  of $\psdG{G}$ containing $x$. Following Algorithm~\ref{alg:cholesky_dual} we let
\[
t \coloneqq 
\begin{pmatrix}
1 & 0 & 0\\
0 & 1 & 0\\
1& 1   & 0\\
\end{pmatrix}
\]
and we may verify that, indeed, $x = tt^\T$ holds.
We also have $I = \{i \mid t_{ii} = 0\} = \{3\}$, so 
according to Theorem~\ref{theo:h_chord_faces}, $\stdFace$ is isomorphic to $\stdFace(H)$, where $H$ is the induced subgraph obtained by removing the vertex with label $3$. 
This automorphism can be obtained by  finding $l \in \LT_{++}(G)$ such that 
\[
lt = e_{I},\text{ where } e_{I} = \begin{pmatrix}
1 & 0 & 0\\
0 & 1 & 0\\
0 & 0   & 0\\
\end{pmatrix}.
\]
One possible solution is 
\[
l = \begin{pmatrix}
1 & 0 & 0\\
0 & 1 & 0\\
-1 & -1  & 1\\
\end{pmatrix}.
\]
Then, $Q_l$ is an automorphism of $\psdG{G}$ mapping $x$ to 
$(lt)(lt)^\T = e_I e_I^* = e_I$. Therefore, 
$Q_l (\stdFace) = Q_l(\minFace(x,\psdG{G})) = \minFace(e_I,\psdG{G}) = \stdFace(H)$ is the principal dual face obtaining by considering the subgraph induced by $\{1,2\}$, i.e., zeroing out the third row and column of elements of $\psdG{G}$.

This leads to the following explicit description of $\stdFace$.
\begin{equation}\label{eq:psdgf}
\begin{aligned}
\stdFace &= \{x \in \psdG{G} \mid (Q_l(x))_{33} = 0 \} = 
\{x \in \psdG{G} \mid (lxl^\T)_{33} = 0 \}\\
& = \{x \in \psdG{G} \mid x_{11} - 2x_{31} + x_{22} - 2x_{32} + x_{33} = 0 \}.
\end{aligned}
\end{equation}

The face $\stdFace$ is also projectionally exposed and following the proof of item~$(i)$ of Theorem~\ref{th:hface},
$\hat l = l^{-1}e_I l$ is a lower triangular matrix such that $Q_{\hat l}$ maps $\psdG{G}$ onto $\stdFace$ and fixes $\stdFace$. We also note that $Q_{\hat l}$ is not self-adjoint.

\begin{remark}[Failure of Theorem~\ref{theo:h_chord_faces} for chordal graphs]\label{rem:failure}
	Let $G$ be the following graph.
	\begin{center}
		\begin{tikzpicture}[main/.style = {draw, circle}] 
		\node[main] (1) {$1$}; 
		\node[main] (2) [right of =1]{$2$}; 
		\node[main] (3) [right of =2]{$3$}; 
		\node[main] (4) [right of =3]{$4$}; 
		\draw (1) -- (2);
		\draw (2) -- (3);
		\draw (3) -- (4);
		\end{tikzpicture}
	\end{center}
	The graph $G$ is chordal but not homogeneous.
	We consider the cone $\psdG{G}$ and we will show that 
	it has a face  linearly isomorphic to $\RR^3_+$.
Let $v \coloneqq (1,1,1,1) \in \RR^4$ and consider the 
	rank~$1$ matrix $vv^\T $, which belongs to $(\psdG{G})^*$.
	Let $\stdFace \coloneqq \psdG{G} \cap \{vv^\T\}^\perp = \{x \in \psdG{G} \mid v \in \ker x\}$. 
	
	Every matrix that generates an extreme ray of $\psdG{G}$ must have its support contained in a maximal clique of $G$, e.g., see \cite[Theorem~2.3]{AHM88}. 
	So if $x$ is an extreme ray of $\psdG{G}$ then it has one of the following patterns of nonzero entries:
	\begin{equation*}
	\begin{pmatrix}
	x_{11} & x_{12} & 0 & 0 
	\\
	x_{12} & x_{22} & 0 & 0 
	\\
	0 & 0 & 0 & 0 
	\\
	0 & 0 & 0 & 0 
	\end{pmatrix},
	\begin{pmatrix}
	0 & 0 & 0 & 0 
	\\
	0 & x_{22} & x_{23} & 0 
	\\
	0 & x_{23} & x_{33} & 0 
	\\
	0 & 0 & 0 & 0 
	\end{pmatrix},
	\begin{pmatrix}
	0 & 0 & 0 & 0 
	\\
	0 & 0 & 0 & 0 
	\\
	0 & 0 & x_{33} & x_{34} 
	\\
	0 & 0 & x_{34} & x_{44} 
	\end{pmatrix}.
	\end{equation*}
	If, in addition, $x$ is an extreme ray of $\stdFace$ then $v$ is in the kernel of $x$ which implies that $x$ must be a positive multiple of one of the following matrices
	\begin{equation*}
	\begin{pmatrix}
	1 & -1 & 0 & 0 
	\\
	-1 & 1 & 0 & 0 
	\\
	0 & 0 & 0 & 0 
	\\
	0 & 0 & 0 & 0 
	\end{pmatrix},
	\begin{pmatrix}
	0 & 0 & 0 & 0 
	\\
	0 & 1 & -1 & 0 
	\\
	0 & -1 & 1 & 0 
	\\
	0 & 0 & 0 & 0 
	\end{pmatrix},
	\begin{pmatrix}
	0 & 0 & 0 & 0 
	\\
	0 & 0 & 0 & 0 
	\\
	0 & 0 & 1 & -1 
	\\
	0 & 0 & -1 & 1
	\end{pmatrix}.
	\end{equation*}
	Each of these three matrices is also in $\stdFace$, so we conclude that 
	$\stdFace$ has exactly three extreme rays and, is therefore, linearly isomorphic to $\RR^3_+$.
	
	The face $\stdFace$ also contains a rank~$3$ matrix (e.g., the sum of these three matrices). In order for a face associated to an induced subgraph of $G$  to contain a rank~$3$ matrix it must be the result of removing a single vertex. However, removing a single vertex of $G$ leads to a face that will, at the very least, contain a face isomorphic to $\psdcone{2}$ (corresponding to one of the remaining edges of the induced subgraph). 
	In particular, such a face is not polyhedral. 
	
	The conclusion is that no face associated to an induced subgraph of $G$ can be linearly isomorphic to $\stdFace$. In particular, $Q$ as in 
	Theorem~\ref{theo:h_chord_faces} may not exist when $G$ is only chordal.
\end{remark}

\subsection{The PSD completion side}
Let $G = (V,E)$ be a homogeneous chordal graph. 
The dual cone of $\psdG{G}$ is the cone of positive semidefinite (PSD) completable matrices that follow the sparsity pattern defined by $G$. That is, $x \in \psdG{G}^*$ if and only if $x \in \pi_G(\psdcone{n})$.
A PSD completion of $x$ is any $w \in \psdcone{n}$ satisfying $x = \pi_G(w)$.
While there is an extensive literature on completion problems, there seems to be few works that address the actual facial structure of $\pi_G(\psdcone{n})$.

Now, suppose that the vertices are ordered according to a trivially perfect elimination ordering. Since $\psdG{G}$ is $ \stdCone(\alg(G))^*$, the cone $\pi_G(\psdcone{n}) =\psdG{G}^*$ coincides  with $\closure \stdCone(\alg(G))$. 
Overall, we have
\[
\pi_G(\psdcone{n}) = \{u\star u^\T \mid u \in \UT_{+}(G)\} = \{\pi_G(u u^\T )\mid u \in \UT_{+}(G)\}.
\] 
An element may have multiple PSD completions and it is of interest to obtain completions that have certain desired properties. 
For example, if there exists a positive definite completion to some $x \in \pi_G(\psdcone{n})$, we may want to obtain one that maximizes the determinant.

In the next theorem, we gather several facts about the facial structure of $\pi_G(\psdcone{n})$ together with a proof 
that Algorithm~\ref{alg:cholesky} leads to a maximum rank decomposition. 
As before, if $H = (\hat{V},\hat{E})$ is an induced subgraph of $G$, then it also corresponds to a face of $\pi_G(\psdcone{n})$ in a natural way. 
By an abuse of notation, we define 
$\pi_H(\psdcone{n}) \coloneqq \{x \in \pi_G(\psdcone{n}) \mid x_{ij} = x_{ji} = 0, \forall i,j \not \in \hat V \}$.
Put otherwise, $x \in \pi_H(\psdcone{n})$ if it has a PSD completion $w$ satisfying $w_{ii} = 0$ for $i \not \in \hat V$.


\begin{theorem}[Faces of homogeneous PSD completable cones]\label{theo:completion}
	Let $G = (V,E)$ be  a homogeneous chordal cone and let 
	$\stdFace \face \pi_G(\psdcone{n})$.
	Then, $\stdFace$ is projectionally exposed and there exists an automorphism 
	$Q$ of  $\pi_G(\psdcone{n})$ that maps $\stdFace$ to $\pi_H(\psdcone{n})$, where $H$ is an induced subgraph of $G$.
	Additionally, if the vertices of $G$ are in a trivially perfect elimination ordering the following items hold.
	
	\begin{enumerate}[$(i)$]
		\item $Q$ can be assumed to be of the form 
		$Q = Q_u$ for some $u \in \UT_{++}(G)$. In particular 
		$Q_u(y)  = \pi_G(uyu^\T )$ holds for $y \in \pi_G(\psdcone{n})$.
		\item Suppose that $x \in \reInt \stdFace$.
		Then, $x = \pi_G(tt^\T)$ holds for a unique proper matrix $t \in \UT_{+}(G)$.
		Then, letting $I = \{i \mid t_{ii} = 0\}$, $H$ is the subgraph obtained by removing the vertices belonging to $I$.
		Furthermore,  $u$ in the previous item can be taken to be any $u \in \UT_{++}(G)$ satisfying $ut = e_I$, where 
		$e_I$ is the diagonal matrix that has $1$ in its $(i,i)$ component if $i \not \in I$ and zero elsewhere.
	\end{enumerate}

	In what follows, let $x,t$ be as in item~$(ii)$ and let  $r 
	\coloneqq |\{i \mid t_{ii} \neq 0\}|$.
	\begin{enumerate}[$(i)$]
		\setcounter{enumi}{2}
		\item The maximum possible rank of a PSD completion of $x$ is $r$ and $tt^\T$ is a maximum rank completion. 
		\item If $r = n$, then $tt^{\T}$ is the maximum determinant completion.
	\end{enumerate}	
\end{theorem}
\begin{proof}
	As in the proof of Theorem~\ref{theo:h_chord_faces}, changing the ordering of the vertices of $G$ amounts to permuting rows and columns, thus leading to linearly isomorphic cones. So we might as well assume 
	that $G$ follows a trivially perfect elimination ordering.
	
	Both the fact that $\stdFace$ is projectionally exposed and that there exists an automorphism mapping $Q$ to some principal faces follow from Theorem~\ref{th:hface}.
	Under a trivially perfect elimination ordering, 
	$Q = Q_u$ for some $u \in \UT_{++}(G)$.
	Then, if $y \in \pi_G(\psdcone{n})$ and 
	$y= t \star t^\T = \pi_G(tt^\T)$ for some $t \in \UT_{++}(G)$, we have $Q_u(y) = (ut) \star (ut)^\T = \pi_G(utt^\T u^\T) = \pi_G(uyu^\T)$, which proves item $(i)$.
	
	Item~$(ii)$ is  a direct consequence of
	the uniqueness of the generalized Cholesky factorization in  Proposition~\ref{prop:chol} and item~$(ii)$ of Theorem~\ref{th:hface}. 
	
	For item~$(iii)$, we start by observing that $\stdFace$ is a homogeneous cone of rank $r$, also by item~$(ii)$ of Theorem~\ref{th:hface}.  
	Then, $\stdFace^\Delta$ is a  homogeneous cone of rank $n-r$, by Proposition~\ref{prop:conj_rank}. 
	Let $y \in \reInt \stdFace^\Delta$ and write 
	$y = ll^\T$, where $l$ is a proper lower triangular matrix in $\LT_{+}(G)$.
	By Theorem~\ref{th:hface_dual}, since $\stdFace^\Delta$ has rank $n-r$, the diagonal of $l$ has exactly $n-r$ positive elements, so $y$ has rank $n-r$.
	Let $w$ be a PSD completion of $x$, then
	\[
	0 = \inProd{y}{x} = \inProd{y}{\pi_G(w)}= \inProd{\pi_G(y)}{w} = \inProd{y}{w},
	\]
	where the last equality follows from $y \in \psdG{G}$. This immediately implies that the rank of $w$ is at most $r$. 
	Analogously, the diagonal of $t$ has exactly $r$ positive elements, so $tt^\T$ has rank $r$ and it is a maximum rank completion of $x$.
	
	For item~{$(iv)$}, since $r = n$, letting $l \coloneqq t^\T$, we note  
	that $l \in \LT_{++}(G)$. 
	Let  $\bar y \coloneqq l^{-1}(l^{-1})^\T$
	and $\bar w \coloneqq tt^\T$.
	In our T-algebra context, 
	$l^{-1}$ is the element in $\LT_{++}(G)$ satisfying 
	$l^{-1} \star l = e$, where $e$ is the identity matrix, see Section~\ref{sec:hom} and the discussion after \eqref{eq:qadj}. However, in view of item~\ref{ltstar} of Theorem~\ref{theo:talg_ch}, $l^{-1} \star l = l^{-1}l= e$ holds, so $l^{-1}$ is, in fact, the usual matrix inverse. 
	Overall, we have  $(tt^\T)^{-1} = l^{-1}(l^{-1})^{\T}$, i.e., $\bar w^{-1} = \bar{y}$ holds in the usual linear algebraic sense.
	Also, $\bar{y} \in \reInt \stdCone(\alg(G))^* = \reInt \psdG{G} $ holds, since the diagonal of $l^{-1}$ is positive.

	We consider the following primal dual pair of problems.
	\begin{equation*}
	\min _{w \in \psdcone{n}, \pi_G(w)=x} -\log \det (w) \qquad \max_{y \in \psdG{G}} \inProd{x}{y} + \log \det (y) + r
	\end{equation*}
	The optimality conditions are
	$\pi_G(w)=x$, $w \in \reInt \psdcone{n}$, 
	$w^{-1} = y$, $y \in \reInt \psdG{G}$. Therefore, $\bar y$ and $\bar w$ are optimal solutions and $\bar w$ is the maximum determinant completion.
\end{proof}
We remark that an analogous result to item~$(iv)$ was proved before in \cite[Appendix~B.4]{TV23}.
Also, Theorem~\ref{theo:completion} implies that $x$ generates an extreme ray of $\pi_{G}(\psdcone{n})$ if and only if the maximum possible rank of its PSD completions is one. 
As it may be of independent interest, we will check that this property  holds when $G$ is only chordal.

\begin{proposition}\label{prop:chord_complet}
	Let $G = (V,E)$ be a chordal graph, $x \in \pi_G(\psdcone{n})$ be nonzero and denote by $r$ the maximum possible rank of a PSD completion of $x$.
	Then, $x$ generates an extreme ray if and only if $r=1$.
\end{proposition}
\begin{proof}
	First, suppose that $x$ generates an extreme ray.
	Let $\stdFace \coloneqq \minFace(x, \pi_G(\psdcone{n}))$, i.e., the minimal face of $\pi_G(\psdcone{n})$ containing $x$.
	Let $\stdFace^\Delta = \psdG{G}\cap \{x\}^\perp$ be the conjugate face and $y \in \reInt \stdFace^\Delta$. 
	Then, if $w$ is a PSD completion of $x$ we have
	\[
	0 = \inProd{x}{y} = \inProd{\pi_G(w)}{y} = \inProd{w}{\pi_G(y)} = \inProd{w}{y}.
	\]
	Because $w$ and $y$ are PSD matrices, we have $wy = yw = 0$ and, in particular, $\ker w \supseteq \text{im}\, y$.
	Also $w \neq 0$, since $x \neq 0$.

	Since $x$ generates an extreme ray, $\stdFace^\Delta$ is a maximum proper face of $\psdG{G}$ by Lemma~\ref{lem:conj}.
	Let $y' \in \psdG{G}$ be such that it generates an extreme ray that is not in $\stdFace^\Delta $. Since $\stdFace^\Delta$ is maximal,
	we have 
	$\minFace(y+y', \psdG{G}) = \psdG{G}$, which implies that 
	$y+y'$ is positive definite and, in particular, has rank $n$.
	Because $G$ is chordal, extreme rays correspond to rank one matrices (\cite[Theorem~2.3]{AHM88}), so $y'$ has rank one, which implies that 
	$y$ has rank~$n-1$.
	
	Finally, since $\ker w \supseteq \text{im}\, y$ holds, 
	we have $\rank w \leq 1$, which in view of $w \neq 0$, leads to $\rank w = 1$. This tells us that $r = 1$.
	
	For the converse, suppose that $r = 1$ holds and $x_1 + x_2 = x$, for nonzero $x_1,x_2 \in \pi_G(\psdcone{n})$.
	Let $w_i$ be a PSD completion of $x_i$, which must also be nonzero because $x_i$ is nonzero. Then 
	$w_1 + w_2$ is a PSD completion of $x$, so 
	$w_1 + w_2$ has rank one.
	This can only happen if $w_1 = \alpha w_2$ for some $\alpha \geq 0$, which implies that 
	$x_1 = \alpha x_2$ and $x_1, x_2$ are in the half-line generated by $x$, so $x$ is an extreme ray.
\end{proof}
Beyond Proposition~\ref{prop:chord_complet}, we do not know if items~$(i)$ and $(ii)$ of Theorem~\ref{theo:completion} have analogues for the case where $G$ is only chordal, although Remark~\ref{rem:failure} may be an indication that an analogue may not exist.

Moving on, as discussed previously, faces of homogeneous cones are projectionally exposed but it may happen that the projection cannot be taken to be self-adjoint. 
We will observe here that if $G$ is chordal, 
then $\pi_G(\psdcone{n})$ is orthogonally projectionally exposed under the trace inner product if and only if $G$ is a disjoint union of cliques. 

\begin{proposition}\label{prop:orth_proj}
Let $G = (V,E)$ be a chordal graph, then the following are equivalent.
\begin{enumerate}[$(i)$]
	\item $\pi_G(\psdcone{n})$ is orthogonally projectionally exposed under the trace inner product.
	\item $\pi_G(\psdcone{n}) = \psdG{G}$ (i.e., $\psdG{G}$ is self-dual).
	\item $G$ is a disjoint union of cliques.
\end{enumerate}
\end{proposition}
\begin{proof}
 We start by observing that 
a pointed cone that is orthogonally projectionally 
exposed must be contained in its  dual, e.g., 
\cite[Proposition~2.2]{BLP87}.	
Therefore, if $\pi_G(\psdcone{n})$ is orthogonally projectionally exposed, then 
$\pi_G(\psdcone{n}) \subseteq \pi_G(\psdcone{n})^* = \psdG{G}$. However, $\psdG{G}\subseteq \pi_G(\psdcone{n})$ always holds, so we have $\pi_G(\psdcone{n}) = \psdG{G}$.
This shows $(i) \Rightarrow (ii)$.
	
If $\pi_G(\psdcone{n}) = \psdG{G}$, then,
since $G$ is chordal, $\pi_G(\psdcone{n})$ is generated by PSD rank one matrices and $G$ must be a disjoint union of cliques by \cite[Theorem~4.2]{GIL24}. 
This shows $(ii) \Rightarrow (iii)$.

Finally, if $G$ is a disjoint union of cliques then, up to permutations of rows and columns, 
$\pi_G(\psdcone{n})$ is a direct product of positive semidefinite cones, so it is orthogonally projectionally exposed under the trace inner product, which shows the implication 
$(iii) \Rightarrow (i)$.
\end{proof}
Proposition~\ref{prop:orth_proj} does not exclude the possibility that some homogeneous chordal $G$ that is not a disjoint union of cliques be such that $\pi_G(\psdcone{n})$ is orthogonally projectionally exposed \emph{under a different inner product}.
Also, we mention in passing that the implication $(ii) \Rightarrow (i)$ is a special case of a more general phenomenon, see \cite[Proposition~4.18]{GL23}.

\section{Conclusion and open questions}\label{sec:conc}
In this paper our main goal was to elucidate the facial structure of homogeneous cones under the T-algebra framework and discuss  applications to homogeneous chordality. 

Here we briefly describe a potential practical application. 
When solving a semidefinite program that fails Slater's condition, one of the default ways of restoring constraint qualifications and reducing the size of the problem is through facial reduction \cite{BW81,Pa13_2,WM13}.
This is done by reformulating the problem over a smaller face of the cone, which leads to a smaller SDP. 
The reformulation typically uses the projectional exposedness of $\psdcone{n}$ and expressions such as \eqref{eq:psd_faces}.  
Unfortunately, this often destroys sparsity if  $q$ is as in \eqref{eq:psd_faces}, since $qxq^\T$ does not necessarily satisfy any sparsity pattern present in $x$. 
Therefore, one must face the inevitable choice of either keeping sparsity or regularizing the problem.
From a theoretical point of view, the issue is that  the map $x \mapsto qxq^\T$ is an automorphism of $\psdcone{n}$ but not necessarily an automorphism of 
the PSD slice that corresponds to the sparsity pattern of the problem.

However, if the sparsity pattern is homogeneous chordal, we have the tantalizing option of both regularizing the problem \emph{and} preserving sparsity. If the feasible region of some SDP is contained in a proper face of a $\psdG{G}$ as in Theorem~\ref{theo:h_chord_faces}, we can use an automorphism of $\psdG{G}$ to reformulate the problem over a principal face, which is an operation that indeed preserves sparsity. 
Furthermore, since principal faces correspond to zeroing rows and columns, we can further reformulate the problem as a smaller dimensional SDP. 
For example, if the feasible region is contained in the face $\stdFace$ as in \eqref{eq:psdgf}, we can reformulate the problem as a $2\times 2$ SDP that still respect the same sparsity pattern by making use of the map $Q_l$  that takes $\stdFace$ to the corresponding rank~$2$ principal face. 
Similar considerations apply to problems over $\pi_G(\psdcone{n})$, in view of Theorem~\ref{theo:completion}. 
It would be interesting to explore these possibilities in future works.

We conclude with some open questions. 

\begin{enumerate}[$(a)$]
	\item  \emph{Suppose that $G = (V,E)$ is chordal but not necessarily homogeneous. Are $\psdG{G}$ and $\pi_G(\psdcone{n})$ projectionally exposed?
	What are the connections between the induced subgraphs of  $G$ and the faces of $\psdG{G}$?} 
	We already saw in Remark~\ref{rem:failure} that, in contrast to homogeneous chordal graphs, $\psdG{G}$ may have faces that are not linearly isomorphic to faces arising from induced subgraphs. Nevertheless, it would be interesting to clarify the extent to which the graph structure of $G$ influences the facial structure of $\psdG{G}$.
	
	\item We proved homogeneous cones are projectionally exposed but they may fail to be orthogonally projectionally exposed \emph{under the inner product that comes from the T-algebra structure}, see Proposition~\ref{prop:orth_proj}.
	We do not know whether an arbitrary homogeneous cone can become orthogonally projectionally exposed by changing the inner product appropriately.
	Currently, the only homogeneous cones known to be orthogonally projectionally exposed are symmetric cones \cite[Proposition~33]{L21}, are those the only ones?
	\item Homogeneous cones are spectrahedral and, therefore, are hyperbolicity cones as well. 
	These are two classes of cones that strictly contain homogeneous cones. Are they projectionally exposed?
	
\end{enumerate}
Regarding (b), this problem is related to finding nontrivial conditions that ensure that a homogeneous cone is actually symmetric, e.g., \cite{YN16}. 
Related to that, motivated by certain considerations in quantum physics that are out-of-scope here,
it was recently shown that if the automorphism group of a homogeneous cone acts transitively on the set of extreme rays, then it must be a symmetric cone \cite[Theorem~2]{BUW23}.
{\small
	\section*{Acknowledgements}
	The support of the Institute of Statistical Mathematics in the form of a visiting professorship for the first author is gratefully acknowledged. 
	We also thank Hideyuki Ishi for several helpful comments and Mitsuhiro Nishijima for carefully reading this manuscript.
}

\bibliographystyle{alpha}
\bibliography{bib}

\end{document}